\documentclass[ejs]{imsart}

\RequirePackage[OT1]{fontenc}
\usepackage{amsthm,amsmath, amssymb, amsfonts, braket, url,enumerate,float,subfig, graphicx,diagbox}
\usepackage{bbm}
\RequirePackage[numbers]{natbib}
\usepackage[]{algorithm2e}
\usepackage{tikz}
\usepackage[utf8]{inputenc} 
\usetikzlibrary{calc,trees,positioning,arrows,chains,shapes.geometric,%
    decorations.pathreplacing,decorations.pathmorphing,shapes,%
    matrix,shapes.symbols}

\RequirePackage[colorlinks,citecolor=blue,urlcolor=blue]{hyperref}

\startlocaldefs
\newcommand{\norm}[1]{\left\lVert#1\right\rVert}
\newcommand{\abs}[1]{\left\lvert#1\right\rvert}
\newcommand{\av}[1]{\left\langle#1\right\rangle}

\def\rank{{\rm rank}}
\newtheorem{theorem}{Theorem}[section]

\newtheorem{lemma}[theorem]{Lemma}

\theoremstyle{definition}
\newtheorem{remark}{Remark}
\theoremstyle{definition}

\DeclareMathOperator*{\argmin}{arg\,min}
\endlocaldefs

\begin{document}

\begin{frontmatter}

\title{Spectral thresholding for the estimation of Markov chain transition operators}
\runtitle{Spectral thresholding for Markov chains}


\begin{aug}


\author{\fnms{Matthias} \snm{L\"offler}\thanksref{a}\ead[label=e2]{matthias.loeffler@stat.math.ethz.ch}}
\and 
\author{\fnms{Antoine} \snm{Picard}\thanksref{d}\ead[label=e4]{ antoine.picard-weibel@ens.fr}}
\thankstext{a}{During the undertaking of parts of this work M. L\"offler was a PhD student at the University of Cambridge and supported by ERC grant UQMSI/647812 and EPSRC grant EP/L016516/1. Further partial support by ETH Foundations of Data Science is gratefully acknowledged. }
\affiliation{ETH Z\"urich \thanksmark{a} and École normale supérieure\thanksmark{d}}

\address[a]{M. L\"offler \\ETH Z\"urich \\Seminar for Statistics \\ 
R\"amistrasse 101, 8092 Zurich, Switzerland  \\
	\printead{e2}}

\address[d]{A. Picard \\ École normale supérieure \\ 45 Rue d'Ulm 75005 Paris, France \\
	\printead{e4}}
\end{aug}
\begin{abstract}
 We consider nonparametric estimation of the transition operator $P$ of a Markov chain and its transition density $p$  where the singular values of $P$ are assumed to decay exponentially fast. This is for instance the case for periodised, reversible multi-dimensional diffusion processes observed in low frequency.
 
We investigate the performance of a spectral hard thresholded Galerkin-type estimator for $P$ and ${p}$, discarding most of the estimated singular triplets. The construction is based on smooth basis functions such as wavelets or B-splines. We show its statistical optimality by establishing matching minimax upper and lower bounds in $L^2$-loss. Particularly, the effect of the dimensionality $d$ of the state space on the nonparametric rate improves from $2d$ to $d$ compared to the case without singular value decay.
\end{abstract}





\end{frontmatter}

\section{Introduction}
We consider an aperiodic and irreducible Markov chain $(X_i )_{i \in \mathbb{N}}$ with the $d$-dimensional torus $\mathbb{T}^d$ as state space. The dynamics of this chain are described by its transition operator,
\begin{equation*}
Pf(x)=\mathbb{E} [f(X_1)|X_0=x]  = \int_{\mathbb{T}^d} f(y) p (x,y) \text{d}y,
\end{equation*}
where $f \in L^2=L^2(\mathbb{T}^d )$. We are interested in nonparametric estimation of the transition density $p(\cdot, \cdot )$ and thus the transition operator $P$, too. \\
Nonparametric estimation of $p$ when assuming smoothness of $p$ has been thoroughly studied, e.g.
\cite{AkakpoLacour, Birge, Clemencon, Lacour, Sart,AthreyaAtuncar98, Roussas69,DoukhanGindes83}. If $p \in H^s$, where $H^s$ denotes the $L^2$-Sobolev space of smoothness $s$, the $L^2$-minimax rates for estimating $p$ are \begin{equation*} \label{Intro Minmax standard} n^{-\frac{s}{2s+2d}}. \end{equation*}
Here we use the additional information provided by assuming that $P$ has an approximately \emph{low rank} structure to improve these rates. Precisely, denoting by $\mu$ the invariant density of $P$ and assuming that $p$ and $\mu$ are bounded, we see that $P$ is a compact operator acting on $L^2(\mu)$ and hence (c.f. Theorem 3.4.1 in \cite{Helemskii06}) has functional singular value decomposition with respect to $\mu$, i.e. 
\begin{equation*}
Pf=\sum_{k \geq 0} \lambda_k \langle u_k, f \rangle_{\mu}v_k~~~~~f \in L^2(\mu),
\end{equation*} 
and we assume that the singular values $\lambda_k$ decay exponentially fast, in the sense that for constants $c, C >0$
\begin{equation*}
\lambda_k \leq C\exp{\left(-c k^{\frac{2}{d}}\right)}.
\end{equation*}
This assumption is motivated by discrete, low frequency observations of periodised, reversible diffusion processes for which it  is fulfilled by virtue of Weyl's law 
\cite{Garding, Hormander,  Ivrii2000, IvriiWeyl, Weyl1911}. Indeed, for a $1$-periodic Lipschitz continuous vector field  $b(x)=(b_1(x), \dots, b_d(x))$, a scalar $1$-periodic $\sigma(x)$ and a standard Brownian motion $W_t$ define the multi-dimensional diffusion process
\begin{equation*} \label{Intro Diffusion 1}
dY_t=b(Y_t)dt+\sigma(Y_t)dW_t,~~~~t \geq 0,
\end{equation*}
and consider its periodised version 
\begin{equation*} \label{Intro Diffusion per}
X_t=Y_t ~~~\text{modulo} ~~~\mathbb{Z}^d, ~~~t \geq 0. 
\end{equation*}
Then $P=P_1$ is one instance of the Feller semigroup $(P_t)_{t \in \mathbb{R}_+}$ with infinitesimal generator $L:H^2 \rightarrow L^2$, and one obtains that $P=\exp(L)$ where
\begin{equation*} \label{Intro Elliptic L} L=\frac{\sigma^2(x)}{2}\sum_{i=1}^d  \frac{\partial^2 }{\partial x_i^2} + \sum_{i=1}^d b_i(x) \frac{\partial }{\partial x_i}. \end{equation*}
$L$ is an elliptic operator and moreover, since the diffusion is assumed to be reversible, $L$ is self-adjoint with respect to the invariant measure $\mu$. For a more thorough explanation and proof of these facts we refer to \cite{BakryGentilLedoux14}. Hence, Weyl's law \cite{Ivrii2000} applies and states that the $k$-th singular value of $L$ is of order $-k^{\frac{2}{d}}$ and thus, as $P=\exp(L)$, the $k$-th singular value of $P$ is bounded by $C\exp({-ck^{\frac{2}{d}}})$. 
\\ \\
Rapid decay of the singular values is also  observed empirically in applications such as molecular dynamics (see e.g. \cite{MaggioniChem}). This has prompted practitioners and applied mathematicians to estimate only the first few singular triplets of $P$ and discard the rest in their analysis \cite{Chodera, Maggioni,Koltai, MaggioniChem, Schutte,SchwantesGibbonPande,  Shukla}. However, often no theoretical guarantees are provided and it is not clear whether their procedures are optimal from a statistical point of view. \\ \\
Low rank assumptions for Markov chains have only recently began to be considered in the statistical literature,  primarily in the finite state case \cite{ZhangLikeli, ZhangWang,DuanWangWenYuan19,ZhuLiWangZhang19}. In these works it is assumed that the transition matrix has a low rank structure and they show nearly optimal rates for their algorithms by keeping only a fixed number of singular triplets or by thresholding singular values. \\ 
After a first version of this paper was uploaded to arXiv, the continuous state space case was also considered by \cite{SunGongDuanWang19}. In contrast to our work, they use a kernel approach and consider the fixed rank case, proving $\ell_2$-error bounds and furthermore providing an algorithm for metastable state clustering. \\ \\ 
We investigate a modified version of one popular method from molecular dynamics for the estimation of $P$, where the number of singular triplets kept  is chosen in a data driven way.  Considering a Galerkin-type estimator \cite{GobetHoffmannReiss, Schutte, SchwantesGibbonPande} we use techniques from low rank matrix estimation and completion  \cite{CandesPlan11,KloppTrace,KoltchinskiiLouniciTsybakov11,CarpentierKim18, CandesPlan10,Cai_Candes_Shen, KloppTrace,NegahbanWainwright,AgarwalNegahbanWainwright12}.  Particularly we show that hard thresholding singular values yields minimax optimal $L^2$-rates 
\begin{equation*}
n^{-\frac{s}{2s+d}} \log(n)^{\frac{d}{2} \frac{s}{2s+d}}
\end{equation*}
over the class of  Markov chains with exponentially decaying singular values. This improves the dependance on the dimension $d$ from $2d$ to almost $d$ compared to the case without singular value decay. Moreover, our analysis reveals that our algorithm keeps at most $C \log(n)^{\frac{d}{2}}$ singular triplets of the estimated transition operator, thus justifying the commonly used approach to discard most of them. Simulations complement our theoretical results and show the improved performance when thresholding singular values.
\section{Main results}\subsection*{Basic Notation }
Let $\mathbb{T}^d$ denote the $d$-dimensional torus, isomorphic to the unit cube $[0,1]^d$ when opposite points are identified, equipped with Lebesgue measure $\lambda$. By $L^2=L^2(\mathbb{T}^d, \lambda)$ we denote the space of square integrable functions (with respect to $\lambda$) on $\mathbb{T}^d$ equipped with Euclidean inner product $\langle \cdot, \cdot \rangle$ and corresponding $L^2$ norm $\| \cdot \|_{L^2}$. We also denote the Euclidean inner product for any finite dimensional vector space by $\langle \cdot, \cdot \rangle$ and the corresponding norm by $\| \cdot \|_2$. For any probability measure $\mu$ on $\mathbb{T}^d$, if $\mu$ has a density with respect to the Lebesgue measure, we denote it in slight abuse of notation by $\mu$, too. Moreover, when considering functions in $L^2\left(\mu\right)=L^2(\mathbb{T}^d, \mu)$, we use the canonical scalar product and denote it by $\av{\cdot, \cdot}_\mu$ with corresponding norm $\| \cdot \|_{L^2(\mu)}$.
$\norm{\cdot}_{{L}^\infty}$ denotes the  ${L}^\infty$ norm. 
$\norm{\cdot}_{F}$ and $\norm{\cdot}_{F, \mu}$ 
denote the Hilbert--Schmidt (Frobenius) norms of operators acting on $L^2$ and $L^2\left(\mu\right)$, respectively, 
while $\norm{\cdot}_{\infty}$ 
denotes the spectral norm for the $\lambda$ 
inner product.  \\
For $s \in \mathbb{N}$ we define the Sobolev space of smoothness $s$ as
\begin{equation*}
H^s:=\{ f \in L^2: \|f\|_{H^s}:=\sum_{|i| \leq s} \| D^i f \|_{L^2 } < \infty \},
\end{equation*}
where $D^{i} \cdot =\partial_{i_1}\dots \partial_{i_d} \cdot $ denotes the partial derivative in direction $i$. 
For $s \notin \mathbb{N}$, $H^s$ is defined through interpolation or equivalently through Fourier methods 
(see Chapter I.9 in \cite{LionsMagenes} or Section 7 in \cite{AdamsFournier}). For $s > 0$ we will also use the H\"older spaces $C^s$ equipped with H\"older norm $\| \cdot \|_{C^s}$. We also employ the same notation for vector fields $f=(f_1, \dots , f_d)$. For example $f \in C^s$ means that  $\| f \|_{C^s}:=\sum_{i} \| f_i\|_{C^s} < \infty$.
We will sometimes use the notation $a \lesssim b$, meaning that $a \leq C b$ for some universal constant $C>0$ which does not depend on $n$.  
\subsection{Assumptions on the model}
For ease of presentation and to avoid boundary issues we assume a periodised setting. However, in principle, our results can be conveyed to general compact state spaces with smooth boundary (and reflected diffusion processes instead of periodized diffusion processes), arguing as in \cite{Nickl20} for arguments involving wavelets and applying results from \cite{Evans} for arguments involving elliptic PDE theory. 

We assume that we observe a Markov chain $\left(X_{i }\right)_{0 \leq i \leq n}$ with state space $\mathbb{T}^d$ and we introduce a set of Markov chains with smoothness index $s>0$ denoted by $\mathcal{M}(s)=\mathcal{M}(s,C_\mu, c_\mu, C_1, C_2, \dots, C_6 )$ fulfilling the following assumptions:
\begin{itemize}
	\item[\textbf{A1}:] $\left(X_{i}\right)_{i \in \mathbb{N}_0}$ is irreducible, aperiodic and has invariant measure $\mu$ which has a density which we will also denote by $\mu$ and $X_0 \thicksim \mu$. 
	\item[\textbf{A2}:] The invariant measure $\mu$ is bounded away from $0$ and $\infty$, i.e. for constants $C_\mu>c_\mu>0$, $c_\mu\leq\mu\leq C_\mu$. 
	\item[\textbf{A3}:]   $\| \mu \|_{H^s} \leq C_1$ for some constant $C_1 >0$.
\end{itemize}
Note that assumption \textbf{A2} implies that $L^2=L^2(\mu)$ and that the pairs of norms  $\norm{\cdot}_{L^2}$ and $\norm{\cdot}_{L^2\left(\mu\right)}$
are equivalent, as well as the induced Hilbert-Schmidt norms (see Lemma \ref{Lemma HS}) and we will therefore use them interchangeably.  

Recall, that the transition operator ${P}$ is defined on $L^2\left(\mu\right)$ by
\begin{equation*}
{P}f (x)= \mathbb{E}\left[ f\left(X_{1}\right)\mid X_0=x\right].
\end{equation*}
We assume that $P$ is an integral operator with kernel $p(x,y)$, the transition density. 
\begin{itemize}
	\item[\textbf{A4}:] $C_2 > p(x,y)>0$ for all $x,y \in \mathbb{T}^d$ and for a constant $C_2 > 1$ . 
\end{itemize}
Since $p$ and $\mu$ are bounded, we have that $p \in L^2(\mu \times \mu)$ and hence $P$ is compact (see Theorem 3.3.1. and Example 1.3.6 in \cite{Helemskii06}). Therefore it has a functional singular value decomposition (e.g. Theorem 3.4.1 in \cite{Helemskii06}): there exist two $\mu$-orthonormal systems $\left(u_k\right)_{k\in \mathbb{N}}$  and $\left(v_k\right)_{k\in \mathbb{N}}$ in ${L}^2\left(\mu\right)$ and a non-negative decreasing sequence $\left(\lambda_k\right)_{k\in \mathbb{N}}$ such that in $L^2(\mu \times \mu)$, 
\begin{align} \label{Assumptions Eigendecomp}
& Pf= \sum_{k} \lambda_k \av{u_k, f}_\mu v_k, ~~~f \in L^2(\mu),\\ &p(x,y)=\sum_k \lambda_k u_k(y)\mu(y)v_k(x).
\end{align}
Having obtained the representation \eqref{Assumptions Eigendecomp} it is thus natural to formulate the remaining assumptions on the singular values and left and right singular functions. We assume that $P$ has an approximately low rank structure with exponential decay of the singular values and that the left and right singular functions obey a certain degree of smoothness.
\begin{itemize}
	\item[\textbf{A5}:] The $k$-th singular value (counting multiplicity) is bounded by \\$C_3 \exp{\left(-C_4 k^{\frac{2}{d}}\right)}$ for positive constants $C_3$ and $C_4$.
	\item[\textbf{A6}:] The absolute spectral gap fulfills $$\gamma:=1-\lambda_1=1-\sup_{f \in L^2(\mu), \langle f, 1 \rangle_{\mu}=0, \|f\|_{L^2(\mu)}=1} \| Pf \|_{L^2(\mu)} \geq C_5$$ for some constant $C_5 >0.$
	\item[\textbf{A7}:] 
	The singular functions $\left(u_k, v_k \right )$ fulfill 
	$\sum_k \lambda_k^2(\norm{u_k}_{H^s}^2+\|v_k\|_{H^s}^2)\leq C_6$ for some constant $C_6> 0$. 
\end{itemize} 
When considering the class $\mathcal{M}(s)=\mathcal{M}(s,C_\mu, c_\mu, C_1, \dots, C_6 )$ we will suppress the dependence on all parameters except $s$, since they, treating them as constants, do not change the minimax rate as long $\mathcal{M}(s)$ has non-empty interior.
We will also write that $p\in \mathcal{M}(s)$ or $P\in \mathcal{M}(s)$ 
if it is the transition density or the transition operator of a Markov chain in $\mathcal{M}(s)$, respectively.  \\ \\
Discrete observations of periodised, reversible  diffusion processes (which have also been considered in \cite{ Kweku, NicklRay, WaaijZanten}) fulfill these assumptions under mild conditions on the volatility function $\sigma$ and drift function $b$ detailed in the Lemma below. This includes for example periodised versions of the Langevin processes considered by Roberts and Tweedie \cite{RobertsTweedie96} or discrete, periodised AR(1) time series. 
\begin{lemma}
	\label{diffusions}
	For a vector field $b(x)=(b_1(x), \dots, b_d(x))$ and a scalar $\sigma(x)$ consider the diffusion process $dY_t=b(Y_t)dt+\sigma(Y_t)dW_t$, $t \geq 0$, and its periodised version $X_t=Y_t ~\text{modulo} ~\mathbb{Z}^d$. Assume that the observations  are given by $(X_i)_{0 \leq i \leq n}$. 
	Moreover, assume that $\sigma(m+x)=\sigma(x)$ and $b(x+m)=b(x)$ for all $m\in \mathbb{Z}^d$ {and that $\sigma^{-2}b= \nabla B$ for some $B \in C^2$}. If $\norm{\sigma^{-2}}_{C^{s-1}}$, $\norm{\sigma^{2}}_{C^{s-1}}$ and $\norm{b}_{C^{s-1}}$ are finite for some $s > 2$, then $p \in \mathcal{M}(s).$ \end{lemma}
The proof of Lemma \ref{diffusions} follows after an application of Weyl's law for operators with non-smooth coefficients due to Ivrii \cite{Ivrii2000} and by using PDE results for elliptic operators  from a recent article by Nickl and Ray \cite{NicklRay} and is presented in the appendix. The proof relies crucially on reversibility of the diffusion process, which allows to relate the singular triplets of the infinitesimal generator $L$ to the singular triplets of $P$.

\subsection{Construction of the estimator}
Here we describe how to obtain estimators for $p$ and $P$ given observations $(X_i)_{0 \leq i \leq n}$,  using a Galerkin approach. This method has also been employed for estimating the drift and volatility functions in a scalar diffusion model in the seminal paper by Gobet et al. \cite{GobetHoffmannReiss}  and the first part of our construction is closely related. \\ 
Instead of estimating $p$ in the functional space, the Galerkin approach estimates the action of $P$ on a suitable approximation space 
and we obtain plug-in estimators for $p$ and $P$. 
%
Working with $P$ 
instead of $p$ is advantageous because we can fully use its low-rank nature.
We construct our estimator as a modified version of the estimator described by Gobet et al. \cite{GobetHoffmannReiss}, adjusted to the non-reversible case: \\
For a periodised Daubechies wavelet basis $\{\Psi_{j,k} \}$ of $L^2(\mathbb{T})$ with regularity greater than $s$ (see e.g. section 4.3. in \cite{GineNickl16} for a construction), extend this basis to a tensor wavelet basis of $L^2(\mathbb{T}^d)$ (see again section 4.3. in \cite{GineNickl16}) given by
$$ \{ \Psi_{j,k}, ~~~k=0, \dots, \max (2^{jd}-1, 0), ~j \in \{-1\} \cup \mathbb{N}_0 \}. $$
 For convenience, we denote this basis $\{\Psi_\lambda\}$ where $\lambda=(j,k)$ is a multi index. 
In the following $V_J$ denotes the linear span of wavelets up to resolution level $J$,
\begin{align*}
V_J:= \text{span}\left\{\Psi_\lambda, ~|\lambda|=|(j,k)|:=j \leq J\right\},
\end{align*}
and we denote by $\textbf{V}_J$ the corresponding space of wavelet coefficients. 
The dimension of $V_J$ is bounded by $C 2^{Jd}$. Moreover, the $\| \cdot \|_{H^s}$-norms are equivalently defined through the decay of wavelet coefficients in this basis, see for instance p. 370 in \cite{GineNickl16}.  
As in Gobet et al. \cite{GobetHoffmannReiss}, we will use bold letters for the coefficient expansions in the wavelet basis $\left(\Psi_\lambda\right)$ of functions and operators in and on $L^2$. These denote vector and matrix like elements. The corresponding functions and operators - which do not depend on the basis - are in italic. In the case of vectors or matrix elements whose coefficients are only defined for $\abs{\lambda}\leq J$, we will sometimes consider them as elements in the whole sequence space. This is done through setting the undefined coefficients to zero.\\  
	Let now $J$ be a resolution level which we will choose later. Following Gobet et al. \cite{GobetHoffmannReiss} we construct a first estimator
	\begin{equation*}
	\left (\hat{\textbf{R}}_{J} \right )_{\lambda, \lambda'}=\frac{1}{n} \sum_{i=0}^{n-1} \Psi_\lambda\left(X_{i}\right)\Psi_{\lambda '}\left(X_{i+1}\right)~~~\text{for}~ |\lambda| \leq J, ~|\lambda'| \leq J.
	\end{equation*}
	The ergodic theorem implies that each of these coefficients converges almost surely to its expectation, 
	\begin{equation*} \mathbb{E}\left[\Psi_\lambda\left(X_0\right)\Psi_{\lambda'}\left(X_1\right)\right]
	=\langle \Psi_\lambda, P \Psi_{\lambda'} \rangle_\mu.
	\end{equation*}
	We thus also introduce $\textbf{R}_J$ which is defined as the expectation of $\hat{\textbf{R}}_J$, i.e.
	\begin{equation*}
	\left(\textbf{R}_J\right)_{\lambda, \lambda'}=\langle \Psi_\lambda, P\Psi_{\lambda'} \rangle_\mu ~~~~\text{for}~| \lambda| \leq J, ~|\lambda'| \leq J.
	\end{equation*}
	As $\overline{\cup_{J\in \mathbb{N}} V_J}=L^2$, we can define $\textbf{R}$, the limit of $\textbf{R}_J$ (with respect to the Hilbert--Schmidt norm). 
	Note that $\textbf{R}$ is defined through the $L^2(\mu)$-inner product
	and therefore, in  general,
	$$\textbf{R}\neq \textbf{P}:=(\langle \Psi_{\lambda}, P \Psi_{\lambda '} \rangle)_{\lambda, \lambda'}.$$ 
	We need to match the scalar products to estimate $P$. Let $G$ be the Gram operator with corresponding sequence representation $\textbf{G}=(\langle \Psi_{\lambda}, G \Psi_{\lambda '} \rangle)_{\lambda, \lambda'}.$
	$G$ is such that $\forall u, v\in L^2$ $\av{u, Gv}=\av{u,v}_\mu$, i.e. it corresponds to pointwise multiplication with $\mu$. Therefore,
	defining $\textbf{u}=(\langle u, \Psi_{\lambda} \rangle)_\lambda$ (and $\textbf{v}$ similarly), we have that
	$$\av{\textbf{u}, \textbf{R} \textbf{v}}=\av{u, P v}_\mu=\av{u, GP v}=\av{\textbf{u}, \textbf{G}\textbf{P} \textbf{v}}.$$
	If we estimate $\textbf{G}^{-1}$ we are thus able to estimate $\textbf{P}$.
	Following Gobet et al. \cite{GobetHoffmannReiss}, we define $$\left(\textbf{G}_J\right)_{\lambda, \lambda'}:= \av{\Psi_\lambda, \Psi_{\lambda'}}_\mu ~~~~\text{for}~~~~ {|\lambda|}\leq J, {|\lambda'|}\leq J$$ and $\hat{\textbf{G}}_J$ as:
	\begin{equation*}
	\left ( \hat{\textbf{G}}_{J}\right )_{\lambda, \lambda'}= \frac{1}{n+1}\sum_{i=0}^{n} \Psi_\lambda\left(X_i\right)\Psi_{\lambda'}\left(X_i\right)~~~\text{for}~ |\lambda| \leq J, ~|\lambda'| \leq J.
	\end{equation*}
	\\
	From here on, our approach differs from that in Gobet et al.  \cite{GobetHoffmannReiss}. In their (reversible) setting, recovering the first non-trivial eigenpair is sufficient, as the drift and volatility functions are identified in terms of this eigenpair and the invariant measure. \\
	Since our objective is to estimate $p$ and $P$ we have to consider $\emph{all}$ singular triplets instead. By assumption $\textbf{A5}$ $P$ has approximately low rank and hence $\textbf{R}_J$, the matrix of projected coefficients of $GP$, is an approximately low rank matrix.  For this reason we use the usual scheme for estimating low rank matrices, see for instance \cite{YuanEkiciLuMonteiro07, Klopp11,BuneaSheWegkamp11, KoltchinskiiLouniciTsybakov11} and hard threshold the singular values of  $\hat{\textbf{R}}_J$. This yields which singular triplets should be discarded in a data driven way. 
	\\We denote the SVD of $\hat{\textbf{R}}_J $ by $$\hat{\textbf{R}}_J= \sum \hat{\lambda}_k \hat{\textbf{u}}_k \hat{\textbf{v}}_k^T,$$
	where $\hat \lambda_k$ denotes the $k$-th singular value of $\hat{\textbf{R}}_J$ and $\hat{\textbf{u}}_k$ and  $\hat{\textbf{v}}_k$ the corresponding singular vectors. We define the spectral hard threshold estimator at level $\alpha$, $\tilde{\textbf{R}}_J=\tilde{\textbf{R}}_J(\alpha)$ as, 
	\begin{equation} \label{Estimator Hard thresh}
	\tilde{\textbf{R}}_J:=  \sum \hat{\lambda}_k \mathbf{1} \left(|\hat{\lambda}_k |> \alpha\right)  \hat{\textbf{u}}_k \hat{\textbf{v}}_k^T.
	\end{equation}
	Finally, we define the estimator for the action of $P$ on $V_J$ as \begin{equation} \label{EstimatorP} \tilde{\textbf{P}}_J:=\hat{\textbf{G}}_J^{-1}\tilde{\textbf{R}}_J.\end{equation}
For $f \in L^2$, we have, in a $L^2$-sense, the relation
	\begin{equation*}
	Pf(x)=\sum_\lambda ({\textbf{P}} \textbf{f} )_\lambda \Psi_{\lambda}(x), 
	\end{equation*}
	and hence we estimate $P$ by $\tilde P$, which we define as  \begin{equation} \label{Estimator P} \tilde{P}f(x):=\sum_{|\lambda|\leq J} (\tilde{{\textbf{P}}}_J \textbf{f} )_\lambda \Psi_{\lambda}(x). \end{equation}
	This also yields an estimator for $p$ by plug-in, given by
	\begin{equation}
	\label{Estimatorp}
	\tilde{p}(x,y):=\sum_{|\lambda| \leq J,~|\lambda'| \leq J} \left (\tilde{\textbf{P}}_J \right )_{\lambda, \lambda'}   \Psi_{\lambda}(x) \Psi_{\lambda'}(y).
	\end{equation}
	We finally choose for a large enough constant $C > 0$ and for $\lceil \cdot \rceil$ denoting the ceiling function,
	\begin{equation} \label{Resolution}
	J = \left\lceil\log_2 (n^{\frac{1}{2s+d}} \log(n)^{-\frac{d}{4s+2d}})\right\rceil ~~~\text{and}~~~\alpha=C\sqrt{\frac{2^{Jd}}{n}},
	\end{equation}
	to obtain the theoretical results in Theorem \ref{Result1 P} in the next section. 

	\subsection{Convergence rates}
	%
	%
	We now give our main theoretical result for the estimator  
	$\tilde p$ of the transition density $p$ constructed in 
	\eqref{Estimatorp}. 
	The upper bounds attained in $L^2$-loss for estimating $p$ match the lower bounds and are therefore minimax optimal, showing that 
	the logarithmic factors are inherent in the information-geometric structure of the problem. \\  
	Comparing our result to the standard Markov chain case  without singular value decay  where the $L^2$ minimax rates are $n^{-\frac{s}{2s+2d}}$ (e.g. \cite{Clemencon, LacourSPA}), we see that the effect of the dimension on the rate improves, up to the logarithmic factor, from $2d$ to $d$. 
	\begin{theorem}
		\label{Result1 P}
		Suppose that we observe $\left(X_i\right)_{0\leq i\leq n}$ drawn from a stationary Markov Chain with $p \in \mathcal{M}(s)$ for some $s \geq d$. Then, for the estimator $\tilde{p}$ defined in \eqref{Estimatorp} and a constant $C > 0$ we have, for $n$ sufficiently large enough, with probability at least 
		$1-6\exp\left(-n^{\frac{d}{2s+d}} \log(n)^{-\frac{d^2}{4s+2d}}\right)$ that
		\begin{equation}
		\label{Result1 p cvgc}
		\norm{p-\tilde{p}}_{L^2}\leq C \log\left(n\right)^{\frac{d}{2}\frac{s}{2s+d}} n^{-\frac{s}{2s+d}}.
		\end{equation}
		Moreover, the following minimax lower bound holds: for constants $c, p_0 > 0$,
		\begin{equation} \label{Result lower bound}
		\inf_{\hat{p}}\sup_{p\in \mathcal{M}(s)} \mathbb{P}_p\left(\norm{p-\hat{p}}_{L^2}\geq c \log\left(n\right)^{\frac{d}{2}\frac{s}{2s+d}} n^{-\frac{s}{2s+d}} \right)\geq p_0>0.
		\end{equation}
	\end{theorem}
	\noindent 
	In addition, by isometry this implies the same upper and lower bounds for estimating $P$ with respect to the Hilbert-Schmidt norm. Slightly adjusting the construction of $\tilde p$ by setting $\tilde p$ to zero if $\hat{\textbf{G}}_J$ has not sufficiently large smallest singular value, it is also possible to obtain a version of the bound \eqref{Result1 p cvgc} in expectation. \\
	The proof of the upper bounds for $\tilde{p}$ in \eqref{Result1 p cvgc} is based on an application of concentration inequalities for Markov chains by Jiang et al. \cite{JiangSunFan18}, combined with an $\epsilon$-net argument to obtain tight bounds for the spectral norm rate of $\hat{\textbf{R}}_J$ and an application of the general theory for rank penalized estimators by Klopp \cite{Klopp11}.  \\
	The lower bound \eqref{Result lower bound} requires different arguments compared to the case without decay. There an application of Assouad's Lemma and flipping coefficients suffices \cite{Clemencon}. Instead, here we adapt an idea by Koltchinskii and Xia \cite{KoltchinskiXia} to our nonparametric setting by using projection matrices to infuse the low rank structure of $P$.
	
	\noindent 
	Additionally, the rank of $\tilde{\textbf{P}}$ in \eqref{EstimatorP} is with high probability bounded by approximately $\log(n)^{\frac{d}{2}}$, implying the same low rank structure for $\tilde P$.  This justifies the approach of practitioners such as \cite{Chodera, Maggioni,  Koltai, SchwantesGibbonPande} to dismiss most singular triplets in their analysis.
	\begin{lemma} \label{Corollary low rank}
		Under the conditions of Theorem \ref{Result1 P}, we have for the estimator $\tilde{P}$ given in \eqref{Estimator P}, for some constant $C>0$, that, on the same event of probability at least $1-6\exp\left(-n^{\frac{d}{2s+d}} \log(n)^{-\frac{d^2}{4s+2d}}\right)$ on which \eqref{Result1 p cvgc} holds,
		\begin{equation}
		\rank(\tilde P) \leq C \log(n)^{\frac{d}{2}}.
		\end{equation}
	\end{lemma}
	\begin{remark}[Other basis functions] \label{Remark fourier}
	The proof of Theorem \ref{Result1 P} requires the Jackson  inequality and the bound  $\| v \|_{L^\infty} \leq C\sqrt{\dim(V_J)} $ for any $v \in V_J$ satisfying $\| v \|_{L^2} \leq 1$. Thus, arguing as in Remark 5 in Chorowski and Trabs \cite{ChorowskyTrabs} the conclusions of Theorem \ref{Result1 P} remain valid for the trigonometric and the B-spline basis if one strengthens the assumptions $\mathbf{A3}$ and $\mathbf{A7}$ to $\|\mu\|_{C^s} \leq c$ and $\sum \lambda_k^2 (\|u_k\|_{C^s}^2+\|v_k\|_{C^s}^2) \leq C$ for some constants $c,C>0$.
\end{remark}
	\begin{remark}[Other low-rank algorithms] Instead of using spectral hard thresholding to improve the estimation performance of $P$, one can also use soft thresholding via the matrix lasso \cite{YuanEkiciLuMonteiro07,KoltchinskiiLouniciTsybakov11} or, for a sufficiently large constant $C$, keep the first $C\log(n)^{\frac{d}{2}}$ singular triplets of $\hat{\textbf{R}}_J$. The upper bound in Theorem \ref{Result1 P} and Lemma \ref{Corollary low rank} remain valid with these adjustments, which follows by appealing to arguments from \cite{KoltchinskiiLouniciTsybakov11}. 
	\end{remark} 
	\begin{remark} [ From $P$ to $P_\tau$] \label{Remark Pt} In molecular dynamics it is often desired to obtain an estimate for the transition operator, $$P_\tau f(x):=\mathbb{E} [f(X_\tau)|X_0=x],~~~~f \in L^2(\mu),$$
		and its transition density $p_\tau$, $\tau >1$, for example for simulating or visualizing the Markov chain at a coarser timescale. \\
		Given the estimator $\tilde{\textbf{P}}$ in \eqref{EstimatorP} and $\tau \in \mathbb{N}$ it is possible to obtain an estimator for $p_\tau$ as follows: if $\tau \leq C \log(n)$ we use the plug-in estimator $(\tilde{\textbf{P}})^{\tau}$ and the induced estimator for $p_\tau$ in \eqref{Estimatorp} and we are able to obtain similar upper bounds as in  Theorem \ref{Result1 P} (up to logarithmic factors).\\ If $\tau > C \log(n)$ it suffices to estimate the invariant density $\mu$ as in this case all singular values of $P_\tau$ except the first one are of smaller order than $1/n$. 
	\end{remark}  
	\begin{remark}[Adaptivity]
		The correct choice of $J$ and thus $\alpha$ depends on the smoothness parameter $s$. In practice $s$ is unknown, but one can use for instance Lepski's method to adapt to $s$. The proof that this works and that the upper bound in Theorem \ref{Result1 P} remains valid for this estimator is a straightforward adaptation of results of Chorowski and Trabs \cite{ChorowskyTrabs}. We also refer to Chapter 8 in \cite{GineNickl16} for further details how to prove convergence rates when using Lepski's method. \\ 
	However, it is well known that the performance of Lepski's method on finite data sets is also highly dependent on the choice of further tuning constants in its definition. From a theoretical perspective there have been two attempts to deal with this issue: the concept of minimal penalty \cite{LacourMassart15} and the bootstrap \cite{ChernozhukovChetverikovKato14AOS}. 
	Another practical possibility for tuning parameter selection would be (generalized) cross-validation \cite{WahbaWold75,Li87,CravenWahba79}. However, the theoretical properties of these methods in our 
non-i.i.d. setting are not clear and further research is needed.  

	\end{remark}
	\noindent
	\noindent
	
	\subsection{Numerical Experiment} \label{Numerical Experiments}
	In this section we illustrate and corroborate our theoretical findings with simulated data from a periodized, one-dimensional Ornstein-Uhlenbeck process. 
	The Ornstein-Uhlenbeck process is given by
	\begin{equation} dY_t=-\theta Y_t dt+ \sigma dW_t, ~~~~t \geq 0,\label{OhrnsteinUhlenbeck}
	\end{equation}

We generate observations at discrete time steps $Y_0, Y_1, \dots, Y_{n}$, simulating exactly from the bivariate Gaussian transition density. Afterwards, we periodize the observations with period $2\pi$. Thus,  applying Lemma \ref{diffusions} we obtain that the periodized process fulfills assumptions \textbf{A1}-\textbf{A7} for arbitrary but fixed $s \in \mathbb{N}.$

 The transition density of the periodized process is then given by  
	\begin{equation*}
	p(x,y)=\sum_{y'=y+2 \pi \mathbb{Z}} 
	\frac{1}{\sqrt{\pi \sigma^2(1-e^{-2 \theta })/\theta}} e^{ -\frac{\theta (y-xe^{-\theta })^2}{\sigma^2(1-e^{-2\theta })}}.
	\end{equation*}
For visualization and computing error bounds, we use $y+2\pi i, ~i \in \{-10, \dots, 10\}$, as the other summands above are negligible. 
	\begin{figure}[H]%
	
		\vspace{0cm}
		\centering
	
		\includegraphics[width=12cm]{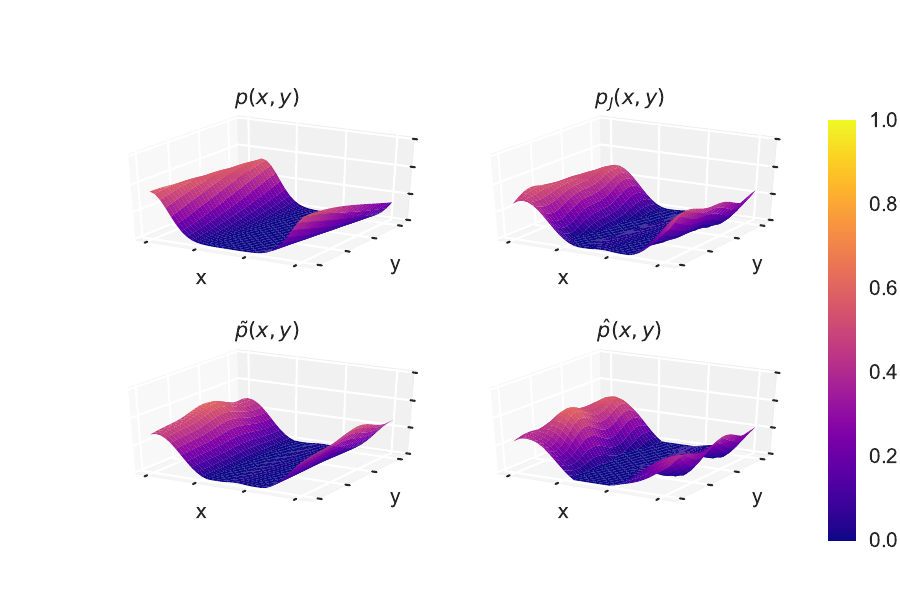} %
		\caption{In clockwise order starting in the upper left corner: Transition-density $p(x,y)$ for the periodized one-dimensional Ornstein-Uhlenbeck process \eqref{OhrnsteinUhlenbeck} with parameters $\theta=2$, $\sigma=2$ and plotted in the region $[0, 2 \pi ]^2$; transition density projected on the approximation space of the first  $J=4$ trigonometric basis functions in each direction; non-thresholded estimator $\hat{p}$ for $n=1000$, $X_0=0.5$ and $J=3$; thresholded estimator $\tilde{p}$ in the same observation scheme with $J=4$ and threshold level $\alpha=0.2$.}
				\label{fig1}
	\end{figure}
	As basis functions, following Remark \ref{Remark fourier}, we use the trigonometric basis on the interval $[0, 2 \pi]$, given by
	\begin{equation*}
	\Psi_k(x)=\begin{cases}
	\frac{1}{\sqrt{2 \pi }}~~~~~~~~~~~~~~~~~~~~~~~~~k=0 \\
	\frac{1}{\sqrt{\pi}}\cos\left ( \frac{x k}{2}  \right )  ~~~~~~~~~~~~~~~k=2i, ~i \in \mathbb{N} \\
	\frac{1}{\sqrt{\pi }}\sin\left ( \frac{x(k+1)}{2}\right )~~~~~~~~~~k=2i-1, ~i \in \mathbb{N}.
	\end{cases}
	\end{equation*}
	We compute the average Euclidean loss for various numbers of observations, resolution levels and thresholding levels, which are displayed in Table \ref{tab:table1}. For the calculation of the Euclidean losses we calculate the first $900$ Fourier coefficients of $p$ and, afterwards, their Euclidean distance to the empirical ones. \\
	Our theoretical findings are supported by the numerical results in Table \ref{tab:table1}. For all number of observations, a hard thresholded estimator achieves better error bounds than non-thresholded estimators. Moreover, compared to the non-thresholded estimator the optimal resolution level $J$ increases, thus allowing the estimation of finer details of the transition densities. \\
	We also visualize this in Figure \ref{fig1} for 1000 observations. There we plot the non-thresholded estimator with smallest error and consequently resolution level $J=3$ against the best thresholded estimator where the optimal resolution level is $J=4$. \\ 
\begin{table}[h!]
	
	\begin{center}

		\begin{tabular}{l|c|r|l}
		\diagbox{\textbf{Parameters}}{\textbf{Observations}} 	 &1000  & 3000%
		  & 6000 \\ 
			\hline
			$J=3, \alpha=0$ & 0.5118 & 0.4722 & 0.4661\\ 
			$J=3, \alpha=0.2$ & 0.5480 &0.5452 & 0.5442 \\ 
			$J=3, \alpha=0.1$ & 0.5473 & 0.5445& 0.5431\\ 
			$J=4, \alpha=0$ &  0.6300 & 0.5659& 0.558\\ 
			$J=4, \alpha=0.2$ & \textbf{0.4701} & 0.4590& 0.4554 \\
			$J=4, \alpha=0.1 $ & 0.4746 & 0.4596 & 0.4547 \\ 
			$J=5, \alpha=0$ & 0.7124 &  0.5002 &  0.4346  \\
		$J=5, \alpha=0.02$ & 0.6351 & 0.4770 & {0.4176} \\
			$J=5, \alpha=0.03$ & 0.6293 & \textbf{0.4224} &  \textbf{0.3928} \\
			$J=5, \alpha=0.05$ & 0.6127& {0.4371} & {0.4094}   \\
			$J=5, \alpha=0.1$ & 0.5033& {0.4611}& 0.4557  \\
			$J=5, \alpha=0.2$ & 0.4965& 0.4631 &  0.4548 \\
			$J=6, \alpha=0$ & 1.0083& 0.6787& 0.6036  \\ 
			$J=6, \alpha=0.05$ & 0.7733 & 0.5834 & 0.5828 \\ 
			$J=6, \alpha=0.1$ & 0.6279 & 0.4783 & 0.4564 \\ 
			$J=6, \alpha=0.2$ & 0.5218 &0.4667 &0.4625  \\ 
	
		\end{tabular}
	~~\\~~  \\ 
			\caption{Average Euclidean loss for the Ornstein-Uhlenbeck process for various resolution and threshold levels. Here we denote by $J$ the number of Fourier coefficients that are computed in each direction. For each observation level we generate $100$ realizations of the estimators, compute the Euclidean loss and calculate the average. As parameters for the Ornstein-Uhlenbeck process we choose $\theta=2$ and $\sigma=2$.  }
					\label{tab:table1}
	\end{center}
\end{table}

\subsection{Discussion}
In this paper we have shown that it is possible to speed up the estimation rates of transition densities of Markov chains in the presence of additional approximately low-rank structure of the corresponding transition operator. We have proposed a new algorithm based on spectral thresholding of a Galerkin-type estimator of the transition operator and proved sharp minimax optimal convergence rates that improve the exponential dependence of the convergence rate on the dimensionality from $s/(2s+2d)$ to almost $s/(2s+d)$. Moreover, we have proven that for a large and natural class of Markov chains, periodized, low frequency observations of diffusion processes, our assumptions are fulfilled. 

Nevertheless, many open questions remain and we outline some possible future research avenues below. 

First of all, it would be interesting to investigate whether our assumptions, particularly exponential decay of the singular values, hold for other classes of Markov chains, e.g. non-reversible diffusion processes or difference equations \cite{Nelson90}  such as the GARCH(1)-model. For non-reversible diffusion processes it would be necessary to consider the singular values of the transition operator directly through the parabolic heat-equation as in the non-reversible case the singular values of $L$ can not be directly related to the singular values of $P$. In case of difference equations, it is known that they converge in the microscopic limit to solutions of diffusion processes \cite{Nelson90} and hence it seems plausible that also difference equations might fulfill our assumptions and, in particular, exponential decay of the singular values of the transition operator under mild assumptions. 

Moreover, the class of Markov chains parameterized through low frequency observations of diffusion processes forms a large and important subclass of $\mathcal{M}(s)$ and underlying diffusion dynamics may be commonly assumed in applications such as molecular dynamics \cite{Chodera, Maggioni,Koltai, MaggioniChem, Schutte,SchwantesGibbonPande,Shukla}. Hence, it would be interesting to investigate whether the bounds in Theorem \ref{Result1 P} are also sharp when  restricted to the  class of diffusion processes considered in Lemma \ref{diffusions}. In order to obtain a lower bound, it would likely be necessary to impose a prior on $(b, \sigma)$ to ensure that $P$ is generated from a diffusion process and then carefully study stability properties of the mapping $(b, \sigma) \rightarrow p$ and thus the (parabolic) heat equation. So far, a similar program has only been successfully  employed in simpler settings of elliptic PDE regression models, e.g. \cite{NicklvdGWang20}. 

Besides estimating the full transition operator $P$, it might also be of interest to estimate the first few $L^2(\mu)$-singular functions $u_k$ and $v_k$, for instance for visualization purposes. This can be done by plug-in and solving the generalized singular value problem $\tilde{\textbf{P}}_J \tilde{\textbf{u}}_{k,J}=\tilde{\lambda}_k \tilde{\textbf{v}}_{k,J}$, $ \textbf{u}_{k,J}, \textbf{v}_{k,J}  \in \textbf{V}_J, \|\tilde{\textbf{G}}_J^{1/2} \textbf{u}_{k, J}\|_2=\|\tilde{\textbf{G}}_J^{1/2} \textbf{v}_{k, J}\|_2=1$.  Combining the spectral norm bounds in Lemma \ref{Lemma Spectral norm bound} with perturbation bounds for generalized eigenvalue problems (e.g. section VI in \cite{StewartSun90}) it is, in principle, possible to obtain $L^2$-estimation bounds for $u_k$ and $v_k$. It is an interesting further research question to investigate whether the bounds obtained in this way are sharp and derive minimax lower bounds for generalized singular value problems in dependence on the smoothness of $\mu$ and $u_k$ and $v_k$ in our Markovian setting, thus extending results for PCA from the i.i.d. setting  \cite{CaiMaWu13,VuLei2013,Wahl21}.

	\section{Proofs}
	Throughout the results and proofs, the constants involved will be denoted by $C$ and $c$; we will not always keep track of them and they may change from equation to equation. However one can check that they can be bounded by functions of constants defining the model in \textbf{A1}-\textbf{A7}.
	\subsection{Upper bounds - proof of \eqref{Result1 p cvgc}}
	\subsubsection{Decomposing the error term}
	We first decompose the error term and then bound each term separately. We have that
	\begin{align}
	\| \tilde{\textbf{P}}_J-\textbf{P}\|_F 
	\leq &   \| \hat{\textbf{G}}_J^{-1} ( \tilde{\textbf{R}}_{J} - \textbf{R}_{r,J})\|_F+\|(\hat{ \textbf{G}}_J^{-1}-\textbf{G}_J^{-1})  \textbf{R}_{r,J}\|_F \notag \\  + & \|\textbf{G}_J^{-1} (\textbf{R}_{r,J}-\textbf{R}_J) \|_F+\| \textbf{G}_J^{-1} \textbf{R}_J-\textbf{P} \|_F \notag \\
	\leq & \| \hat{\textbf{G}}_J^{-1}\|_\infty  \| \tilde{\textbf{R}}_{J} - \textbf{R}_{J}\|_F+ (\|\textbf{G}_J^{-1}\|_\infty+\|\hat{\textbf{G}}_J^{-1}\|_\infty) \|\textbf{R}_{r,J}-\textbf{R}_J \|_F \notag \\ & +r^{1/2}\|(\hat{ \textbf{G}}_J^{-1}-\textbf{G}_J^{-1})  \textbf{R}_{r,J}\|_\infty+\| \textbf{G}_J^{-1} \textbf{R}_J-\textbf{P} \|_F \notag \\
	=:&I+II+III+IV, \label{Proof Main decomp}
	\end{align}
	where $\textbf{R}_{r,J}$ will denotes a rank-$r$ approximation of $\textbf{R}_{J}$ which we define below and where we used that $(\hat{ \textbf{G}}_J^{-1}-\textbf{G}_J^{-1})  \textbf{R}_{r,J}$ has rank at most $r$ to obtain the term III. We therefore have to take care of $4$ terms: Variance bounds in Frobenius norm (I), rank-$r$ approximation error (II), correction of the scalar product in spectral norm (III), and smoothness approximation error (IV). 
	\subsubsection{Bounding I - variance bounds in spectral and Frobenius norm}
	In this section we bound the first term $\| \hat{\textbf{G}}_J^{-1}\|_\infty  \| \tilde{\textbf{R}}_{J} - \textbf{R}_{J}\|_F$. We will first obtain a bound for $\|\tilde{\textbf{R}}_J-\textbf{R}_J\|_F$. In our proof, we follow the usual line of arguments from the low rank literature \cite{Klopp11,KoltchinskiiLouniciTsybakov11} and bound the spectral norm of $\hat{ \textbf{R}}_J-\textbf{R}_J$. Moreover, we also prove spectral norm bounds for  $ \hat{ \textbf{G}}_J-\textbf{G}_J$ and $ (\hat{ \textbf{G}}_J-\textbf{G}_J)\textbf{P}_J$, where $\textbf{P}_J$ denotes the restriction of $\textbf{P}$ to $\textbf{V}_J$. 	\begin{lemma}
		\label{Lemma Spectral norm bound}
		Assume $2^{3Jd}\leq cn$ for some small enough constant $c>0$. Then for constants $C, C', C''>0$ we have that
		\begin{align}
		& \mathbb{P} \left ( 	\norm{\hat{ \textbf{R}}_J-\textbf{R}_J}_\infty \leq C\sqrt{\frac{2^{Jd}}{n}} \right ) \geq 1- 2\exp\left(-2^{Jd} \right ),
		\label{Bound R_J spectral norm}
		\\
		& 	\mathbb{P} \left ( 	\norm{\hat{ \textbf{G}}_J-\textbf{G}_J}_\infty \leq C'\sqrt{\frac{2^{2Jd}}{n}} \right ) \geq 1- 2\exp\left(-2^{Jd} \right ),
		\label{Bound G_J spectral norm} \\
		&	\mathbb{P} \left ( 	\norm{(\hat{ \textbf{G}}_J-\textbf{G}_J) \textbf{P}_J }_\infty \leq  C''\sqrt{\frac{2^{Jd}}{n}} \right ) \geq  1-2\exp\left(-2^{Jd} \right ).
		\label{Bound G_JPspectral norm}
		\end{align}
	\end{lemma}
	\noindent \begin{proof} 
		We only prove \eqref{Bound R_J spectral norm} as the two other bounds follow from the same argument. \\
		We use an $\epsilon$-net argument, arguing exactly as in the proof of Lemma 1.1 in Cand\`es and Plan \cite{CandesPlan11}. Indeed, arguing as in \cite{CandesPlan11} we have, since  $\textbf{V}_J$ has dimension $C2^{Jd}$, that there exists a $\frac{1}{4}$-net $D_\frac{1}{4}$ of the unit sphere in $\textbf{V}_J$ for Euclidean distance of cardinality less than $9^{C 2^{Jd}}$. \\
		Now let $\textbf{v}$ and $\textbf{u}$ with $\|\textbf{u}\|_2=\|\textbf{v}\|_2=1$ such that $\norm{\hat{ \textbf{R}}_J-\textbf{R}_J}_\infty=\textbf{v}^T (\hat{ \textbf{R}}_J-\textbf{R}_J ) \textbf{u}$ and $\textbf{u}_0$ and $\textbf{v}_0$ contained in $D_{\frac{1}{4}}$ such that $\| \textbf{u}-\textbf{u}_0\|_2 \leq 1/4, ~\| \textbf{v}-\textbf{v}_0\|_2 \leq 1/4$. We obtain that
		\begin{align*}
		& \norm{\hat{ \textbf{R}}_J-\textbf{R}_J}_\infty \\  =&  \langle \textbf{v}_0, (\hat{ \textbf{R}}_J-\textbf{R}_J) \textbf{u}_0 \rangle + \langle \textbf{v}-\textbf{v}_0, (\hat{ \textbf{R}}_J-\textbf{R}_J) \textbf{u} \rangle + \langle \textbf{v}_0, (\hat{ \textbf{R}}_J-\textbf{R}_J)( \textbf{u}-\textbf{u}_0 )\rangle  \\
		\leq  & \frac{1}{2} \norm{\hat{ \textbf{R}}_J-\textbf{R}_J}_\infty + \langle \textbf{v}_0, (\hat{ \textbf{R}}_J-\textbf{R}_J) \textbf{u}_0 \rangle
		\end{align*}
		and hence it suffices to bound $ \textbf{v}^T  \left( \hat {\textbf{R}}_J-\textbf{R}_J\right)\textbf{u}$ on $D_{\frac{1}{4}}$. Generalizing formula (24) in Lemma 19 in Nickl and S\"ohl \cite{NicklSohl} by using a Bernstein inequality for non reversible Markov chains by Jiang et al. \cite{JiangSunFan18} we obtain Lemma \ref{Pointwise bound} which can be found in the appendix. Applying Lemma \ref{Pointwise bound} and  using a union bound we obtain,
		\begin{align*}
		& \mathbb{P}\left( \norm{\left(\hat{ \textbf{R}}_J-\textbf{R}_J\right)}_{\infty}  > 2 C\sqrt{\frac{2^{Jd}}{n}} x\right) \\
		\leq & \mathbb{P}\left(\max_{\textbf{u}_0, \textbf{v}_0\in D_\frac{1}{4}}  \langle \textbf{v}_0, (\hat{ \textbf{R}}_J-\textbf{R}_J) \textbf{u}_0 \rangle \ > C\sqrt{\frac{2^{Jd}}{n}} x\right) 
		\leq 
		2 \cdot 9^{2C2^{Jd}}\left( e^{-2^{Jd}x}\right).
		\end{align*}
		Applying this with $x=1+2C\ln\left(9\right)$ finishes the proof of \eqref{Bound R_J spectral norm}. For the other two bounds, we use the same argument as above,	appealing to  the bounds \eqref{Lemma Bernstein G} and \eqref{Lemma Bernstein G sharp}, respectively, instead.

	\end{proof} 
	We continue with bounding $\|\tilde{\textbf{R}}_J-\textbf{R}_J\|_F$. 
	Throughout the rest of the proof we work on the event \begin{align}
	\Omega:=\bigg\{ 	\notag & \norm{\hat{ \textbf{R}}_J-\textbf{R}_J}_\infty \leq C\sqrt{\frac{2^{Jd}}{n}}, ~\norm{\hat{ \textbf{G}}_J-\textbf{G}_J}_\infty \leq C'\sqrt{\frac{2^{2Jd}}{n}},\\ & \norm{(\hat{\textbf{G}}_J-\textbf{G}_J)\textbf{P}_J}_\infty \leq C''\sqrt{\frac{2^{Jd}}{n}} \bigg \}
	\label{Omega}
	\end{align}
	which happens by Lemma \ref{Lemma Spectral norm bound} with probability at least $1-6e^{-2^{Jd}}$. \\
	We now prove Frobenius norm bounds by applying Theorem 2 (iii) by Klopp \cite{Klopp11}. For completeness, we briefly present her proof below. \\
	As noted by Bunea et. al. \cite{BuneaSheWegkamp11} the hard threshold estimator \eqref{Estimator Hard thresh} is the solution of the rank penalized problem
	\begin{equation} \label{Rank penalized estimator}
	\tilde{\textbf{R}}_J = \argmin_{\textbf{S}} \| \hat{\textbf{R}}_J-\textbf{S} \|_F^2 + \alpha^2 \rank(\textbf{S}). 
	\end{equation}
	We suppose that the constant in the definition of $\alpha$ is large enough such that $\alpha\geq 2 C\sqrt{2^{Jd}/n}$. Since $\tilde{\textbf{R}}_J$ is the minimizer of \eqref{Rank penalized estimator} the first inequality holds for any $\textbf{S}$ and afterwards we use that \\ $\langle A, B \rangle \leq \sqrt{\rank(A)} \|A\|_F \|B\|_\infty$ and that $2ab \leq a^2+b^2$ to obtain that 
	\begin{align*}
	& \| \tilde{\textbf{R}}_J-\textbf{R}_J\|_F^2  \leq \|\textbf{S}-\textbf{R}_J\|_F^2 +2 \langle \hat{\textbf{R}}_J-\textbf{R}_J,\tilde{\textbf{R}}_J-\textbf{S} \rangle +\alpha^2 (\rank(\textbf{S})-\rank(\tilde{\textbf{R}}_J))  \\
	\leq & \|\textbf{S}-\textbf{R}_J\|_F^2+\alpha\sqrt{\rank(\tilde{\textbf{R}}_J)+\rank(\textbf{S})} \|\tilde{\textbf{R}}_J-\textbf{S}\|_F +\alpha^2 (\rank(\textbf{S})-\rank(\tilde{\textbf{R}}_J)) \\
	\leq & \|\textbf{S}-\textbf{R}_J\|_F^2+{\alpha}\sqrt{\rank(\tilde{\textbf{R}}_J)+\rank(\textbf{S})} \|\tilde{\textbf{R}}_J-\textbf{R}_J\|_F \\
	&+ {\alpha}\sqrt{\rank(\tilde{\textbf{R}}_J)+\rank(\textbf{S})} \|{\textbf{R}}_J-\textbf{S}\|_F+\alpha^2 (\rank(\textbf{S})-\rank(\tilde{\textbf{R}}_J)) \\
	\leq & \frac{3}{2}\|\textbf{S}-\textbf{R}_J\|_F^2+\frac{1}{2}\| \tilde{\textbf{R}}_J-\textbf{R}_J\|_F^2 + 2 \alpha^2\rank(\textbf{S}).
	\end{align*}
	Summarizing, rearranging terms, we have that on $\Omega$
	\begin{align}
	\label{Lounici}
	I& =\|\hat{\textbf{G}}_J^{-1}\|_\infty \norm{\tilde{\textbf{R}}_J-\textbf{R}_J}_F\lesssim \| \hat{\textbf{G}}_J^{-1} \|_\infty \inf_{\textbf{S}\in\textbf{V}_J \times \textbf{V}_J}\left(\norm{\textbf{S}-\textbf{R}_J}_F^2+ \alpha^2  \rank\left(\textbf{S}\right)\right)^{1/2} \notag \\
	& \leq  \| \hat{\textbf{G}}_J^{-1} \|_\infty \left ( \norm{\textbf{R}_{r,J}-\textbf{R}_J}_F^2+r \alpha^2 \right )^{1/2}
	\end{align}
	We now find the adequate $\textbf{R}_{r,J}$ in \eqref{Lounici} . 
	\subsubsection{Bounding II - low rank approximation error}
	By construction of the extension of operators on $V_J$ as operators in the sequence space we have that, $\textbf{R}_J= {\boldsymbol{\pi}}_J^\lambda \textbf{G} \textbf{P}{\boldsymbol{\pi}}_J^\lambda$, where ${\boldsymbol \pi}_J^{\lambda}$ is the orthogonal projection on $\textbf{V}_J$ 
	with respect to the Euclidean scalar product. 
	\\For any rank $r$ approximation $P_r$ of $P$, $\textbf{R}_{r, J}:= {\boldsymbol{\pi}}_J^\lambda \textbf{G} \textbf{P}_r{\boldsymbol{\pi}}_J^\lambda$ is a rank $r$  approximation of $\textbf{R}_J$ and fulfills $\norm{\textbf{R}_{r, J}-\textbf{R}_J}_F\lesssim  \norm{P_{r}-P}_F \lesssim \|P_r-P\|_{F,\mu}$ as $\|G\|_\infty \leq \|\mu\|_{L^\infty} \lesssim 1$ by \textbf{A2} and by Lemma \ref{Lemma HS}. We define a rank $r$ approximation of $P$ as follows:
	\begin{equation} \label{Low rank approx}
	P_{r}f:=\sum_{k=0}^{r-1} \lambda_k \av{u_k, f}_\mu v_k ~~\text{for}~ f \in L^2(\mu). 
	\end{equation}
	This provides a sequence of approximations $\textbf{R}_{r, J}$ of $\textbf{R}_J$ satisfying \begin{equation} \label{Low rank approx R} \norm{\textbf{R}_{r,J}-\textbf{R}_J}_F^2\lesssim \norm{P_r-P}_{F, \mu}^2 = \sum_{k\geq r} \lambda_k^2. \end{equation} We recall that by assumption \textbf{A5} $\lambda_k\leq C_3 \exp\left(-C_4 k^{2/d}\right)$. Denote by $\lceil \cdot \rceil$ the ceiling function and set \begin{equation} \label{Low rank approx r} r:= \left \lceil C\log \left ({\frac{1}{\alpha}} \right )^{\frac{d}{2}} \right \rceil +2 \end{equation} for $C >0$ large enough. With this choice we obtain that
	\begin{align} \label{Bound low rank approx}
	\norm{\textbf{R}_{r, J}-\textbf{R}_J}_F^2 &\lesssim \sum_{k\geq r} \lambda_k^2
	\lesssim \int_{2C_4r^{\frac{1}{d}}}^\infty x^{d-1}\exp\left(- \frac{x^2}{2} \right) \text{d}x.
	\end{align}
	If $d\geq 3$, we use integration by parts
	\begin{align*}
	F_d\left(y\right):=\int_{y}^\infty x^{d-1}\exp\left(- \frac{x^2}{2} \right) \text{d}x= y^{d-2} \exp\left(- \frac{y^2}{2}\right)+&\left(d-2\right) \int_{y}^\infty x^{d-3}\exp\left(- \frac{x^2}{2} \right) \text{d}x\\
	=y^{d-2} \exp\left(- \frac{y^2}{2}\right)+&\left(d-2\right) F_{d-2}\left(y\right),
	\end{align*}
	and it remains to bound $F_d$ for $d=2$ and $d=1$.  For $y \geq 1$ we have that $F_{1}(y) \leq F_2(y)=\exp(-y^2/2)$  and therefore, by choice of $r$, we obtain overall that

	\begin{equation} \label{bound low rank approx}
	\norm{\textbf{R}_{r,J}-\textbf{R}_J}_F^2 \lesssim\left(\log\frac{1}{\alpha}\right)^{\frac{d}{2}} \alpha^2.
	\end{equation}
	Since  $\rank(\textbf{R}_{r, J})=r$, \eqref{Lounici} implies that on $\Omega$ 
	\begin{align} \label{FrobPJ}
	\norm{\tilde{\textbf{R}}_J- \textbf{R}_J}_F^2
	& \lesssim \alpha^2 \left(\log{\frac{1}{\alpha}}\right)^{\frac{d}{2}}
	\lesssim {\frac{2^{Jd}}{n}} \left (\log{n} \right )^\frac{d}{2}.
	\end{align}
	\subsubsection{Bounding III - correction of the scalar product}
	In this section we bound the third term, $r^{1/2}\|(\hat{ \textbf{G}}_J^{-1}-\textbf{G}_J^{-1})  \textbf{R}_{r,J}\|_\infty $, in the decomposition \eqref{Proof Main decomp}. Moreover, we prove that $\| \hat{\textbf{G}}_J^{-1}\|_\infty \lesssim 1$ on $\Omega$. \\
Since by \textbf{A2}  the invariant density is bounded away from zero, we have  that \begin{align*}
\inf_{\textbf{u} \in \textbf{V}_J} \langle \textbf{u},\textbf{G}_J \textbf{u} \rangle & =  \inf_{u \in V_J} \langle u, G_J u \rangle = \inf_{u \in V_J} \langle u, G u \rangle  =\inf_{u \in V_J} \int_{\mathbb{T}^d} u^2(x)\mu(x) \text{d}x \geq c_\mu \|\textbf{u}\|_2^2. \end{align*}  Hence, when viewed as a $\dim(\textbf{V}_J) \times \dim(\textbf{V}_J)$-matrix, the smallest eigenvalue of $\textbf{G}_J$ is lower bounded by $c_\mu$ and consequently we obtain that $\|\textbf{G}_J^{-1}\|_\infty \leq c_\mu^{-1}\lesssim 1$.  Moreover, on the event $$\left \{ \norm{\hat{\textbf{G}}_J-{\textbf{G}_J}}_\infty\leq c_\mu/2 \right \} \supset \Omega,$$ we have by the triangle inequality $	\forall \textbf{u} \in \textbf{V}_J$ that 
	\begin{align*}
\norm{\hat{\textbf{G}}_J \textbf{u}}_{2}\geq \norm{\textbf{G}_J\textbf{u}}_{2}-\norm{\textbf{G}_J-\hat{\textbf{G}}_J}_\infty\norm{\textbf{u}}_{2} \geq \frac{c_\mu}{2} \|\textbf{u}\|_2
	\end{align*}
and hence on $\Omega$ \begin{equation} \label{Bound hat G_J-1}\| \hat{\textbf{G}}_J^{-1}\|_\infty \lesssim 1 ~~~\text{and}~~~\|\textbf{G}_J^{-1}\|_\infty\lesssim 1.  \end{equation} 
	%
	Moreover, due to the bounds \eqref{Bound G_J spectral norm}, \eqref{Bound G_JPspectral norm}, \eqref{bound low rank approx},  \eqref{FrobPJ} and \eqref{Bound hat G_J-1} and using the identity $\textbf{G}_J^{-1} -\hat {\textbf{G}}_J^{-1}=\textbf{G}_J^{-1} (\textbf{G}_J-\hat{\textbf{G}}_J)\hat{\textbf{G}}_J^{-1}$ we obtain that 
	\begin{align*}
	& \|(\hat{ \textbf{G}}_J^{-1}-\textbf{G}_J^{-1})  \textbf{R}_{r,J}\|_\infty   \lesssim   \|\hat{ \textbf{G}}_J^{-1}-\textbf{G}_J^{-1}\|_\infty \| \textbf{R}_{r,J}-\textbf{R}_J\|_F+\|(\hat{ \textbf{G}}_J^{-1}-\textbf{G}_J^{-1})  \textbf{R}_{J}\|_\infty \\
	\lesssim & \|\textbf{G}_J^{-1}\|_\infty  \|\hat{\textbf{G}}_J^{-1}\|_\infty  \| \hat{ \textbf{G}}_J-\textbf{G}_J\|_\infty \| \textbf{R}_{r,J}-\textbf{R}_J\|_F+\|(\hat{ \textbf{G}}_J^{-1}-\textbf{G}_J^{-1})  \textbf{R}_{J}\|_\infty \\
	\lesssim &
	\| \hat{ \textbf{G}}_J-\textbf{G}_J\|_\infty \left ( \| \textbf{R}_{r,J}-\textbf{R}_J\|_F+
	\| \textbf{G}_J^{-1}\textbf{R}_J-\textbf{P}_J\|_F \right ) +\|(\hat{ \textbf{G}}_J-\textbf{G}_J )\textbf{P}_J\|_\infty \\
	\lesssim & \sqrt{\frac{2^{2Jd}}{n}} \cdot \left ( \sqrt{\frac{2^{Jd}}{n}} \log(n)^{\frac{d}{4}}+IV\right )+\sqrt{\frac{2^{Jd}}{n}} .
	\end{align*}
	\subsubsection{Bounding IV - bias bounds}
	It is left to bound the term IV in \eqref{Proof Main decomp}. \\
We denote by $\pi_J^\lambda$ and $\pi_J^\mu$ the orthogonal projectors on ${V}_J$ for the $\lambda$ and $\mu$ scalar products respectively. \cite{GobetHoffmannReiss} remarks that the non-zero eigenpairs of $\pi_J^\mu P\pi_J^\lambda$ and $G_J^{-1} R_J$ are identical, where  $G_J^{-1}$ denotes the pseudo-inverse of $G_J$.  We quickly prove this here for completeness. 
	\begin{lemma}
		We have the equality
		\begin{align*}
		\pi_J^\mu=(\pi_J^\lambda G \pi_J^\lambda)^{-1} \pi_J^\lambda {G}
		\end{align*}
		which implies that 
		\begin{equation} {G}_J^{-1}{R}_J= \left(\pi_J^\lambda {G} \pi_J^\lambda\right)^{-1} \pi_J^\lambda{R} \pi_J^\lambda=\pi^\mu_J {P}\pi^\lambda_J.
		\label{Bias equality}
		\end{equation}
	\end{lemma}
	\noindent Indeed, $\pi_J^\mu$ minimizes
	\begin{equation*}
	\|{G}^{1/2} (I-\pi_J^\mu \pi_J^\lambda) \|_F^2 ,
	\end{equation*}
	leading to the normal equation
	\begin{align*}
	\pi_J^\lambda {G} (I- \pi_J^\lambda \pi_J^\mu )=0~~& \implies \pi_J^\lambda {G}  \pi_J^\lambda \pi_J^\mu = \pi_J^\lambda {G} \\
	& \implies \pi_J^\mu = (  \pi_J^\lambda {G}  \pi_J^\lambda)^{-1} \pi_J^\lambda {G},
	\end{align*}
	where $(  \pi_J^\lambda {G}  \pi_J^\lambda)^{-1}$ denotes the pseudo-inverse of  $\pi_J^\lambda {G}  \pi_J^\lambda$. 
	\qed \\
	Using this identity, we establish the bias bounds.
	\begin{lemma}
		\label{Bias bound}
		The bias satisfies~:
		\begin{equation}
		\norm{ \textbf{G}_J^{-1}\textbf{R}_J - \textbf{P}}_F
		\lesssim~2^{-Js}.
		\end{equation}
	\end{lemma}
	\begin{proof}
		Note that  $(I-\pi_J^\mu)=(I-\pi_J^\mu)(I-\pi_J^\lambda)$ and that $\|I-\pi_J^\mu\|_\infty \lesssim 1$ arguing as in the proof of Lemma 4.4. in \cite{GobetHoffmannReiss}: indeed, by assumption $\textbf{A2}$ on $\mu$ and Lemma \ref{Lemma HS}   we have that 
		\begin{align*}
		\|\pi_J^\mu\|_\infty & = \sup_{g:~\|g\|_{L^2} \leq 1} \|\pi_J^\mu g\|_{L^2}  \leq \|\mu^{-1}\|_{L^\infty}^{1/2}\sup_{g:~\|g\|_{L^2(\mu)}\leq 1} \|\pi_J^\mu g\|_{L^2}\\ & \leq   \|\mu^{-1}\|_{L^{\infty}} \sup_{g:~\|g\|_{L^2(\mu)}\leq 1} \|\pi_J^\mu g\|_{L^2(\mu)} \lesssim 1.
		\end{align*}
		Therefore, we obtain that 
		\begin{align*}
		\norm{ \textbf{G}_J^{-1}\textbf{R}_J - \textbf{P}}_F  &  = \| \pi_J^\mu P \pi_J^\lambda -P\|_F \leq \|(I-\pi_J^\mu)P\|_F+\|P(I-\pi_J^\lambda)\|_F \\
		&  \lesssim\|(I-\pi_J^\lambda)P\|_F+\|P(I-\pi_J^\lambda)\|_F.
		\end{align*}
If the $\mu$-orthonormal system of right singular functions $\{u_k\}_{k \in \mathbb{N}}$ does not form a $\mu$-orthonormal basis of $L^2(\mu)$, we extend it, by Zorn's Lemma and the Gram-Schmidt process, to a $\mu$-orthonormal basis of $L^2(\mu)$ which we also denote by $\{u_k\}_{k \in \mathbb{N}}$. Applying Lemma \ref{Lemma HS} again and using that, after possible extension,  $\{u_k\}_{k \in \mathbb{N}}$ is a basis of $L^2(\mu)$, we obtain that
		\begin{align*}
		\| (I-\pi_J^\lambda)P\|_F^2 & \lesssim \| (I-\pi_J^\lambda)P\|_{F, \mu}^2= \sum_k \| (I-\pi_J^\lambda)Pu_k \|_{L^2(\mu)}^2 \\ &  = \sum_k \lambda_k^2 \| (I-\pi_J^\lambda)v_k \|_{L^2(\mu)}^2	 \lesssim \sum_k \lambda_k^2 \| (I-\pi_J^\lambda)v_k \|_{L^2}^2 \\ & \lesssim
		\sum_k \lambda_k^2 \|v_k\|_{H^s}^2 2^{-2Js}  \lesssim  2^{-2Js},
		\end{align*}
		where we used Jackson's inequality (e.g. Proposition 4.3.24 in \cite{GineNickl16}) and assumption \textbf{A7}.
		Finally, we bound
		\begin{align*}
		\|P(I-\pi_J^\lambda )\|_F^2 & = \sum_{\Psi_{\lambda} \notin V_J} \|  P \Psi_{\lambda}\|_{L^2}^2 \\ & \lesssim \sum_{\Psi_{\lambda} \notin V_J} \|  \sum_k \lambda_k v_k   \langle u_k, \Psi_{\lambda} \rangle_\mu\|_{L^2(\mu)}^2 \\
		& = \sum_k \lambda_k^2 \|(I-\pi_J^\lambda)(\mu u_k)\|_{L^2}^2 \lesssim 2^{-2Js} 
		\end{align*}
		where we used that, for $s > d/2$, $H^s$ is a Banach algebra and Jackson's inequality and \textbf{A7} again. 
	\end{proof}
	\subsubsection{Rates of convergence for $\tilde{\textbf{P}}$}
	Taking all the above bounds together we obtain on the event $\Omega$ which happens with probability at least $1-6e^{-2^{Jd}}$ that
	\begin{align}
	\| \tilde{\textbf{P}}-\textbf{P} \|_F \lesssim \left ( \log(n)^{\frac{d}{4}} \sqrt{\frac{2^{Jd}}{n}} + 2^{-Js} \right ) \left (1+\sqrt{\frac{2^{2Jd}}{n}} 
	\right )
	\end{align}
	and hence, choosing the optimal resolution and threshold levels from \eqref{Resolution} and since $s>d/2$, we obtain that with probability at least $1-6\exp\left(-n^{\frac{d}{2s+d}} \log(n)^{-\frac{d^2}{4s+2d}}\right)$
	\begin{equation}
	\norm{\textbf{P}-\tilde{\textbf{P}}}_F\lesssim  \log\left(n\right)^{\frac{d}{2}\frac{s}{2s+d}}n^{-\frac{s}{2s+d}}.
	\label{P rates}
	\end{equation}
	\noindent
	The identification between $P$ and $\textbf{P}$ is isometric, and therefore, this proves the rates for estimation of $P$ 
	Moreover, the correspondence between $P$ and $p$ is also isometric, and thus the estimator $\tilde{p}$ achieves the same $L^2$-rates as in \eqref{P rates} on the same high probability event. This ends the proof of \eqref{Result1 p cvgc} in Theorem \ref{Result1 P}. 
	\hfill $\square$
	\subsection{Proof of Lemma \ref{Corollary low rank}} We work throughout
	on the event $\Omega$ defined in \eqref{Omega} where the results of Theorem \ref{Result1 P} hold. Moreover, we have that $\rank(\tilde P)=\rank(\tilde{\textbf{P}})\leq \rank(\tilde{\textbf{R}}_J)=:\tilde r$. Since $\tilde{\textbf{R}}_J$ is a hard thresholding estimator, we have, by Lidski's inequality and denoting by $\lambda_k(\textbf{R}_J)$ the $k$-th singular value of $\textbf{R}_J$ that 
	\begin{align*}
	\tilde r \geq k \implies \lambda_k(\textbf{R}_J)+\|\textbf{R}_J-\hat{ \textbf{R}}_J \|_\infty > 2C\sqrt{\frac{2^{Jd}}{n}}.
	\end{align*}
	On the other hand, on $\Omega$ we have that $\|\textbf{R}_J-\hat{ \textbf{R}}_J \|_\infty \leq C\sqrt{2^{Jd}/n}$. Finally, note that as in \eqref{Bound low rank approx} we have that for some small enough $c >0$
	$$ \lambda_k(\textbf{R}_J)^2 \lesssim \sum_{l \geq k} \lambda_l^2 \lesssim \exp(-ck^{\frac{2}{d}}).$$
	Hence, for $k = C' \log(n)^{\frac{d}{2}}$ for some $C'>0$ large enough, we have that $$\lambda_k(\textbf{R}_J) \leq C\sqrt{\frac{2^{Jd}}{n}}.$$
	Thus, $\lambda_k(\hat{\textbf{R}}_J)$ and the preceding singular values are set to zero by the hard thresholding procedure, implying that $\tilde r \lesssim \log(n)^{\frac{d}{2}}$. 
	\hfill $\square$
	\subsection{Lower bounds - proof of \eqref{Result lower bound}}
	In this section, we prove the minimax lower bounds showing that the rates attained by our estimator are optimal. \\ \\
	We first construct a sufficiently rich sub-set $M \subset \mathcal{M}(s)$ of transition densities. 
	Let $\pi_0$ be the $\lambda$-orthogonal projector onto constants. Let $\left(\Psi_\lambda\right)_{\lambda}$ be a $s$-regular orthonormal periodic wavelet family with at least one vanishing moment and compactly supported. Let $\left(N_J\right)$ be for each $J$ a maximal subset of wavelets of resolution $J$ such that two different wavelets in $N_J$ have disjoint support. We have that $\abs{N_J}\geq c 2^{Jd}$. Let $W_J= \text{span}\left({\Psi\in N_J}\right)$. \\
	Let $\mathcal{G}_{k,J}$ denote the set of all $k$-dimensional subspaces of $W_J$. For every element $S\in \mathcal{G}_{k,J}$, we denote $\pi_S$ the orthogonal projector from $L^2$ to $S$, and define $P_S=\pi_0+ \eta \varepsilon_n \pi_S$, with $$\varepsilon_n=  \left( \log n \right)^{-\frac{d}{4}\frac{d}{2s+d}} n^{-\frac{s}{2s+d}}$$
	and for $\eta > 0$ a constant. The following lemma shows that these $P_S$ are contained in $\mathcal{M}(s)$ for an appropriate choice of $k$ and $J$:
	\begin{lemma}
		Choose $k$ and $J$ such that
		\begin{align*}
		\frac{c_k}{2 }\left(-\log\varepsilon_n\right)^{\frac{d}{2}}\leq k&\leq c_k \left(-\log\varepsilon_n\right)^{\frac{d}{2}}\\
		\frac{c_J}{2}\log(n)^{-\frac{d}{2} \frac{1}{2s+d}} n^{1/(2s+d)} \leq 2^{J}&\leq c_J \log(n)^{-\frac{d}{2} \frac{1}{2s+d}} n^{1/(2s+d)}.
		\end{align*} Then for any choice of constants defining $\mathcal{M}(s)$ such that $\mathcal{M}(s) \neq \emptyset$, we can choose positive constants $c_k$ and $c_J$, such that for $n$ large enough $\forall S \in \mathcal{G}_{k, J}$ $P_S$ is contained in $\mathcal{M}(s)$. 
		\label{Sub Model}
	\end{lemma}
	\begin{proof}
		We carefully check that \textbf{A}\textbf{1}-\textbf{A}\textbf{7} are fulfilled. \\
		We first check $\textbf{A}\textbf{1}$-\textbf{A4} together. 
		Let $b=\left(f_i \right)_{1\leq i\leq k}$ be an orthonormal basis of $S$. Complete it into $\overline{b}=\left(f_i \right)_{1\leq i\leq \abs{N_J}}$ an orthonormal basis of $W_J$ and let $\textbf{f}_{i, \lambda} = \av{f_i, \Psi_\lambda}$ be the change of coordinate matrix between  $\left(\Psi_\lambda\right)_{\lambda\in R_J}$ and $\overline{b}$. Then
		\begin{align*}
		p_S\left(x, y\right)
		=&1+ \varepsilon_n \eta \sum_{i=1}^k \sum_{\lambda\in R_J}\sum_{\lambda'\in R_J}  \textbf{f}_{i,\lambda} \Psi_\lambda\left(x\right) \textbf{f}_{i,\lambda'} \Psi_{\lambda'}\left(y\right)
		\end{align*}
		Note that this formula implies that $\lambda$ is the invariant measure and thus $\textbf{A1}-\textbf{A3}$ once we have proved that $p_S$ defines a probability density.  Since the $\Psi_\lambda$ have disjoint support,
		\begin{align*}
		1-C \eta 2^{Jd}\varepsilon_n \leq p_S\left(x,y\right)\leq 1+ C \eta 2^{Jd}\varepsilon_n.
		\end{align*}
		Since $s \geq d$, $2^{Jd}\varepsilon_n$ goes to $0$ as $n$ grows, implying that for any $c>0$, for $n$ large enough, $0 < 1-c \leq   p_S\left(x, y\right)\leq 1+c$. Moreover, $p$ integrates to $1$ and hence $p$ is indeed a probability density and \textbf{A1}-\textbf{A4} follow.
		Moreover, by definition of $P_S$ the first eigenvalue is $1$, the next $k$ eigenvalues are $\eta \varepsilon_n$ and the remaining eigenvalues are zero. With our choices of $k$ and $\varepsilon_n$ we thus obtain $\textbf{A5}$. Likewise $\textbf{A6}$ is fulfilled as the spectral gap is precisely $1-\eta \varepsilon_n$ which can be made arbitrary close to one. Finally, by the relation $\| f_i \|_{H^s} \leq C2^{Js} \|f_i\|_{L^2}$ which holds for arbitrary $f_i \in W_J$ (see Equation 4.166 and following in chapter 4.3.6 in \cite{GineNickl16}) we obtain that
		\begin{align*}
	1+	\sum_i \lambda_i^2 \norm{f_i}_{H_s}^2 \leq  1 + Ck\eta^2 \varepsilon_n^2 2^{2Js} \leq C
		\end{align*}
		for $n$ large enough and thus $\textbf{A7}$ holds.
		%
		%
	\end{proof}
	We now choose a maximal subset $M$ of $\mathcal{G}_{k,J}$ such that for any two projections in $M$, denoted by $S_1$ and $S_2$ we have that,  \begin{equation} \label{Lower bound distance} \|p_{S_1}-p_{S_2}\|_{L^2}=\norm{P_{S_1}-P_{S_2}}_F\geq c_0 \varepsilon_n \sqrt{k} \end{equation} for a constant $c_0 > 0$. By Proposition 8 in \cite{Pajor} we have for some universal constants $c, C >0$ that,
	\begin{equation} \label{Pajor}
	\left(\frac{c}{c_0}\right)^{k \left(\abs{N_J}-k\right)}\leq \abs{M}\leq \left(\frac{C}{c_0}\right)^{k \left(\abs{N_J}-k\right)}.
	\end{equation}
	We finally add the element $p_0=1$ to $M$. \\ \\

	We now apply Theorem 2.5 in \cite{Tsybakov} and check that its conditions are fulfilled for our choices of $k$ and $\varepsilon_n$. 
	For $p_S \in M$ denote by $\mathcal{P}_S^n$ the probability measure for the Markov chain $(X_0, \dots, X_n)$ with transition density $p_S$ and invariant measure $1$.
	We first show that we can control the Kullback--Leibler divergence $K(\mathcal{P}_S^n, \mathcal{P}_0^n)$ defined for two probability measures $\mathcal{P}$ and $\mathcal{Q}$ with densities $\text{d}\mathcal{P}$ and $\text{d}\mathcal{Q}$ respectively as,
	\begin{equation*}
	K(\mathcal{P},\mathcal{Q}):=\begin{cases} \int_{\mathbb{T}^d}  \log \left ( \frac{\text{d}\mathcal{P}(x)}{\text{d}\mathcal{Q}(x)}\right ) \text{d}\mathcal{P}(x)  ~~~~\mathcal{P} ~\text{is absolutely continous with respect to} ~\mathcal{Q} \\
	\infty
	~~~~~~~~~~~~~~~~~~~~~~~~~~~~~\text{else}
	\end{cases}
	\end{equation*}
	by the squared $L^2$ norm of $p_S-p_0$ \begin{equation*}
	\label{Kullback}
	K\left(\mathcal{P}_S^n, \mathcal{P}_0^n\right)\leq n\norm{p_S-p_0}_{L^2}^2.
	\end{equation*}
	Indeed,
	\begin{align*}
	K\left(\mathcal{P}_S^n, \mathcal{P}_0^n\right)=&\mathbb{E}_{\mathcal{P}_S^n} \left[\log\left(\frac{\text{d}\mathcal{P}_S^n\left(X_0, X_1,... X_n\right)}{\text{d}\mathcal{P}_0^n\left(X_0, X_1, \dots  X_n\right)}\right)\right]\\
	=& \mathbb{E}_{\mathcal{P}_S^n} \left[\log\left(\frac{p_S\left(X_0, X_1\right)\dots p_S\left(X_{n-1}, X_n\right)}{p_0\left(X_0, X_1\right) \dots  p_0\left(X_{n-1}, X_n\right)}\right)\right]\\
	=&n\mathbb{E}_{\mathcal{P}_S^1} \left[\log\left(\frac{p_S\left(X_0, X_1\right)}{p_0\left(X_0, X_1\right)}\right)\right]. 
	\end{align*}
	Further evaluating the last equation we find,
	\begin{align*}
	\mathbb{E}_{\mathcal{P}_S^1} \left[\log\left(\frac{p_S\left(X_0, X_1\right)}{p_0\left(X_0, X_1\right)}\right)\right] = \int_{x} \int_{y} \log\left(p_S\left(x, y\right)\right) p_S\left(x, y\right)\text{d}x\text{d}y.
	\end{align*}
	We can decompose $p_S=1+ \varepsilon_n H_b$. Then, since $\log\left(1+\varepsilon_n H_b\right)\leq \varepsilon_n H_b$, we have that
	\begin{align*}
	\mathbb{E}_{\mathcal{P}_S^1} \left[\log\left(\frac{p_S\left(X_0, X_1\right)}{p_0\left(X_0, X_1\right)}\right)\right] &\leq \int_{x} \int_{y} \varepsilon_n H_b\left(x,y\right) \left(1+\varepsilon_n H_b\left(x,y\right)\right)\text{d}x\text{d}y\\
	&=\int_{x} \int_{y} \varepsilon_n^2 H_b\left(x,y\right)^2\text{d}x\text{d}y\\
	&= \norm{p_0-p_S}_{L^2}^2 = \eta^2 \varepsilon_n^2 \| \pi_S\|_F^2=\eta^2 \varepsilon_n^2k
	\end{align*}
	Thus, ordering the elements $p_S \in M$ from $0$ to $|M|$ with $p_0=1$ and denoting by $\mathcal{P}_i^n$ the respective probability measure for the chain $(X_0, \dots, X_n)$, we obtain that
	\begin{equation*}
	\frac{1}{|M|}\sum_{j=1}^{\abs{M}} \text{K}\left(\mathcal{P}_j^n, \mathcal{P}_0^n\right)\leq n \eta^2\varepsilon_n^2 k.
	\end{equation*}
	The bound \eqref{Pajor} on $|M|$ and our choices of $k$ and $J$ described in Lemma \ref{Sub Model}  then imply
	\begin{equation*}
	n \eta^2 \varepsilon_n^2k \leq  k \left( |N_J|-k)\right) \log(\frac{c}{c_0}) \leq \log\abs{M},
	\end{equation*}
	by choosing $\eta$ small enough.

	Thus, using also \eqref{Lower bound distance}, all conditions of Theorem 2.5 in \cite{Tsybakov} are met and we obtain \eqref{Result lower bound}. Moreover, by isometry the same lower bound holds for $P$. \hfill $\square$ 
	\section{Appendix}
	\subsection{Proof of Lemma \ref{diffusions}}

	The condition \(\sigma^{-2} b = \nabla B\) for some $B\in C^2$ implies, by Theorem 4.2 in \cite{Kent}, that the chain $X_t$ is reversible with invariant measure satisfying $\mu \propto e^{B}$. This identity and the bounds on the $C^{s-1}$ norms of $b$ and $\sigma^{-2}$ imply $\mu \in H^s$ and that $c \leq \mu \leq  C$ for constants $c, C >0$. Moreover, irreducibility and aperiodicity follow by the upper and lower bounds on $p$ below and thus $\textbf{A1}-\textbf{A3}$ are fulfilled. 
	Assumption \textbf{A4} follows by estimates for the heat kernel, see e.g. Theorem 1.1 in \cite{Norris} and by noting that $\sum_{x'=x+\mathbb{Z}^d} Ce^{-c\| x' - y \|_2^2}$ is summable for every $x,y \in \mathbb{T}^d$. Also note that these estimates yield $p(x,y) > c >0$ uniformly for $x, y \in \mathbb{T}^d$. \\
	Assumption \textbf{A5} is implied by Weyl's law for elliptic operators with non-smooth coefficients on closed manifolds, Theorem 3.1. in \cite{Ivrii2000}. Particularly, \textbf{A5} follows by inverting formula (3.4) in \cite{Ivrii2000} applied to the operator $\tilde L=G^{-1/2}LG^{1/2}$ where $L$ is the  infinitesimal generator $L$
	$$L=\frac{\sigma^2(x)}{2} \sum_{i=1}^d \frac{\partial^2}{\partial^2 x_i}+ \sum_{i=1}^d b_i(x) \frac{\partial}{\partial x_i}$$
	(with $m=1$ there) and by noting that the $L^2{(\mu)}$-eigenvalues of $L$ equal the $L^2(\lambda)$-eigenvalues of $\tilde L$ and that the $L^2(\mu)$-eigenvalues of $P$ equal the exponentiated $L^2(\mu)$-eigenvalues of  $L$. \\
	\textbf{A6} follows from arguing as \cite{Kweku} in the proof of Theorem 6, using exercise 7 on p. 493 in \cite{Bhattacharya} instead of the cited Lemma 2.3 there and the lower bound on $p$ from above. \\
	We now show that assumption \textbf{A7} is fulfilled. 
	Adapting Lemma 11 in \cite{NicklRay} to our situation with non-constant but scalar $\sigma$ is straightforward and we obtain that there exists a  $C=C(\|\sigma^{-2}\|_{C^{s-1}},  \|b\|_{C^{s-2}})>0$ such that for all $f\in L^2$ with $\mathbb{E}\left[ f \left( X_0 \right) \right]=0$ we have for $t\leq s$ that \begin{equation*}
	\label{Sobolev Control}
	\norm{L^{-1}(f)}_{H^t}\leq C \norm{f}_{H^{t-2}}, ~~~~\norm{(L^*)^{-1}(f)}_{H^t}\leq C \norm{f}_{H^{t-2}}
	\end{equation*}
	where $L^{-1}(f)$ denotes the solution $u$ to the inhomogeneous p.d.e. $Lu=f$.

	Since $P$ and $L$ are self-adjoint the left and right singular functions coincide, are called eigenfunctions, and we denote them by $e_k$. 
	Since $\av{e_k, 1}_\mu=0$ for $k>0$ we can use this repeatedly for the eigenfunctions $e_k$  which fulfill $L e_k = \log(\lambda_k) e_k$. This  implies that \[\norm{e_k}_{H^s}\lesssim \abs{\log{\lambda_k}}^{\lceil s/2\rceil} \| e_k\|_{L^2} \lesssim \abs{\log{\lambda_k}}^{\lceil s/2\rceil} \lesssim k^{\frac{s+2}{d}},  \]
	where the last inequality follows by using Weyl's law again. Therefore we obtain that, 
	\begin{equation*}
	\sum_k \lambda_k^2 \|e_k\|_{H^s}^2 \lesssim \sum_k k^{\frac{2s+4}{d}} e^{-ck^{\frac{2}{d}}} \lesssim 1
	\end{equation*}
	and \textbf{A7} follows. \qed

	\subsection{Lemma \ref{Pointwise bound}}
	\begin{lemma}\label{Pointwise bound}
		Assume $2^{3Jd}\leq n$ and that $\kappa n 2^{-3Jd} \geq x \geq 1$ for some constant $\kappa >1 $. Then for constants $C=C(\kappa), C'=C'(\kappa)$ and $C''=C''(\kappa)$ and $\forall u,v \in V_J$ with $\|u\|_{L^2}=\|v\|_{L^2}=1$ the three following bounds hold:
		\begin{align}
		&	\mathbb{P}\left( \textbf{v}^T  \left( \hat {\textbf{R}}_J-\textbf{R}_J\right)\textbf{u} > C \sqrt{\frac{2^{Jd}}{n}} x\right) \leq 2e^{-2^{Jd}x}
		\label{weak24gen} \\
		\label{Lemma Bernstein G}
		&	\mathbb{P}\left( \textbf{v}^T  \left( \hat {\textbf{G}}_J-\textbf{G}_J\right)\textbf{u} > C' \sqrt{\frac{2^{2Jd}}{n}} x\right) \leq  2e^{-2^{Jd}x} \\
		\label{Lemma Bernstein G sharp}
		&	\mathbb{P}\left( \textbf{v}^T  \left( \hat {\textbf{G}}_J-\textbf{G}_J\right) \textbf{P}_J \textbf{u} > C'' \sqrt{\frac{2^{Jd}}{n}} x\right) \leq 2e^{-2^{Jd}x}.
		\end{align}
	\end{lemma}
	In each case the proof is an application of the Bernstein type inequality in Theorem 1.1 by Jiang et al. \cite{JiangSunFan18}. 
	Also note that the proof is similar to the proof of Lemma 19 in Nickl and S\"ohl \cite{NicklSohl} but that they use a different concentration inequality.
	We prove \eqref{weak24gen} carefully and only sketch the proofs of the remaining two inequalities as they follow along the same line of argumentation. \\
	Without loss of generality assume that $n$ is even. We use the identity
	\begin{align*}
	&\textbf{v}^T{\left(\hat{\textbf{R}}_J-\textbf{R}_J\right)\textbf{u}}
	={\frac{1}{n}\sum_{i=0}^{n-1} \left(v\left(X_{i}\right)u\left(X_{i+1}\right)-\mathbb{E}\left[v\left(X_0\right)u\left(X_1\right)\right]\right)} \\
	= & {\frac{1}{n}\sum_{i=0}^{n/2-1} \left(v\left(X_{2i}\right)u\left(X_{2i+1}\right)-\mathbb{E}\left[v\left(X_0\right)u\left(X_1\right)\right]\right)} \\  & +{\frac{1}{n}\sum_{i=0}^{n/2-2} \left(v\left(X_{2i+1}\right)u\left(X_{2i+2}\right)-\mathbb{E}\left[v\left(X_0\right)u\left(X_1\right)\right]\right)}
	\end{align*}
	where $v(x)=\sum_\lambda \textbf{v}_\lambda\Psi_{\lambda}(x) $ and $u$ is defined likewise.
	We only treat the first term in the equation above as the second one can be bounded with the same arguments. 
	By Lemma 24 in \cite{NicklSohl} the invariant density of the chain $(X_{2i}, X_{2i+1})_{i \in \mathbb{N}_0}$ is $$\mu_2(x_1,x_2)=\mu(x)p(x,y).$$ Moreover, denoting by $P_2$ the transition operator of $(X_{2i}, X_{2i+1})_{i \in \mathbb{N}_0}$, we can bound its absolute spectral gap by the absolute spectral gap of the original chain $(X_i)_{i \in \mathbb{N}_0}$ by applying Lemma 24 in \cite{NicklSohl}, i.e. for any $f \in L^2(\mu_2)$, $\langle f, 1 \rangle_{\mu_2}=0$, we have that
	$$ \| P_2f \|_{L^2(\mu_2)} \leq \lambda_1 \|f \|_{L^2(\mu_2)}.$$ 
	We upper bound the variance \begin{align*}V_{v,u} :&=\norm{v\left(x\right)u\left(y\right)-\mathbb{E}\left[v\left(X_0\right)u\left(X_1\right)\right] }_{L^2\left(\mu_2\right)}^2 \\
	& \leq \int v(x)^2u(y)^2\mu(x)p(x,y)dxdy \leq \| \mu\|_{L^\infty} \|p\|_{L^\infty} \leq C 
	\end{align*}
	for some constant $C > 0$. Next we bound \begin{align*}\norm{v\left(x\right)u\left(y\right)-\mathbb{E}\left[v\left(X_0\right)u\left(X_1\right)\right] }_{L^\infty} & \leq 2 \norm{v\left(x\right)u\left(y\right) }_{L^\infty} \\
	& \leq C' 2^{Jd}.
	\end{align*}
	We now apply Theorem 1.1 by \cite{JiangSunFan18} (with $\epsilon=x \sqrt{\frac{2^{Jd}}{n}}$ for some constant $\sqrt{n}2^{-3Jd/2} \geq x \geq 1$, $\sigma^2\leq C$ and $c=C' 2^{Jd}$ there) to obtain overall that for some constants $\tau, \tau'>0$
	\begin{align*}
	\mathbb{P} \left ( \textbf{v}^T{\left(\hat{\textbf{R}}_J-\textbf{R}_J\right)\textbf{u}} > x \sqrt{\frac{2^{Jd}}{n}}\right ) \leq \exp \left ( \frac{-x^2 2^{Jd}}{\tau+\tau'\frac{x2^{3Jd/2}}{\sqrt{n}}}\right ).
	\end{align*}
	Using also the assumption $2^{3Jd}\leq  n$ this yields for another constant $\tau''>0$ that 
	\begin{align*}
	\mathbb{P} \left ( \textbf{v}^T{\left(\hat{\textbf{R}}_J-\textbf{R}_J\right)\textbf{u}} > x \tau''\sqrt{\frac{2^{Jd}}{n}}\right ) \leq \exp \left ( {-x 2^{Jd}}\right ) .
	\end{align*}
	For the proof of \eqref{Lemma Bernstein G} note that we have the equality
	\begin{align*}
	\textbf{v}^T  \left( \hat {\textbf{G}}_J-\textbf{G}_J\right)\textbf{u}=\frac{1}{n+1} \sum_{i=0}^n v(X_i)u(X_i)-\mathbb{E}v(X_0)u(X_0).
	\end{align*}Hence, it remains to bound the variance and obtain a pointwise bound. We have that
	\begin{align*}
	\|v(x)u(x)-\mathbb{E}v(X_0)u(X_0)\|_{L^2(\mu)}^2 \leq \int v(x)^2u(x)^2\mu(x) \text{d}x \lesssim 2^{Jd}
	\end{align*}
	and 
	\begin{align*}
	\|v(x)u(x)-\mathbb{E}v(X_0)u(X_0)\|_{L^\infty} \lesssim 2^{Jd}.
	\end{align*}
	For the proof of \eqref{Lemma Bernstein G sharp} we argue as before, this time working with the equality
	\begin{align*}
	\textbf{v}^T  \left( \hat {\textbf{G}}_J-\textbf{G}_J\right) \textbf{P}_J\textbf{u}=\frac{1}{n+1} \sum_{i=0}^n v(X_i)\tilde{u}(X_i)-\mathbb{E}v(X_0)\tilde{u}(X_0),
	\end{align*}
	where $\tilde u=P_Ju$. As above it remains to bound the variance and obtain a pointwise bound. In this case we have that 
	\begin{align*}
	\|v(x)\tilde{u} (x)-\mathbb{E}v(X_0)\tilde{u} (X_0)\|_{L^2(\mu)}^2 \leq \int v(x)^2\tilde{u}(x)^2\mu(x) \text{d}x \lesssim \|\tilde{u}(x)\|_{L^\infty}.
	\end{align*}
	Moreover, denoting by $p_J$ the $L^2$-projection of $p$ to $V_J \times V_J$, we have by Young's convolution inequality 
	\begin{align*}
	\|\tilde u\|_{L^\infty} = \| \int p_J(\cdot, y ) u(y) dy \|_{L^{\infty}} \leq \|u\|_{L^2} \|  p_J\|_{L^2}  \lesssim 1 
	\end{align*}
	Thus, we obtain that
	\begin{align*}
	\|v(X_0)\tilde{u}(X_0)-\mathbb{E}v(X_0)\tilde{u}(X_0)\|_{L^\infty} \lesssim 2^{Jd/2}. 
	\end{align*}
	\qed

\subsection{Lemma \ref{Lemma HS}}
\begin{lemma} \label{Lemma HS}
    Suppose that for positive constants $c_\mu$ and $C_\mu$ we have that $c_\mu \leq \mu \leq C_\mu$. Then, we have that the corresponding Euclidean norms and induced Hilbert-Schmidt norms are equivalent, i.e. for any $f \in L^2$ and any linear operator $T \in L^2(\lambda \times \lambda)$
    \begin{align*}
       &  C_\mu^{-1} \| f\|_{L^2(\mu)}^2 \leq \| f \|_{L^2}^2 \leq c_\mu^{-1} \|f\|_{L^2(\mu)}^2 \\
       & C_\mu^{-1} \|T\|_{F,\mu} \leq \|T\|_{F} \leq c_{\mu}^{-1} \|T\|_{F,\mu} ,
    \end{align*}
    with constants independent of $f$ and $T$. 
\end{lemma}
\begin{proof}
    For the first assertion, we have that
    \begin{align*}
        \|f\|_{L^2}^2 = \int_{\mathbb{T}^d} f^2(x) \frac{\mu(x)}{\mu(x)} \text{d}x \leq \|\mu^{-1}\|_\infty \|f\|_{L^2(\mu)}^2 \leq c_\mu^{-1} \|f\|_{L^2(\mu)}^2
    \end{align*}
    and likewise we obtain
    \begin{align*}
        \|f\|_{L^2(\mu)}^2 =\int_{\mathbb{T}^d} f^2(x) \mu(x) \text{d}x \leq \|\mu\|_\infty \|f\|_{L^2}^2 \leq C_\mu \|f\|_{L^2}^2.
    \end{align*}
    For the second assertion, denote by $G$ the Gram operator,  $\forall u,v \in L^2$, $\langle u, Gv \rangle=\langle u, v \rangle_\mu$, i.e. $G$ applies pointwise multiplication with $\mu$, and note that $\|T\|_{F, \mu}=\|G^{1/2} T G^{1/2}\|_F$ and hence
    \begin{align*}
        \|P\|_{F, \mu} &  \leq \|G^{1/2}\|_{\infty}^2\|T\|_F =\sup_{f:\|f\|_{L^2} = 1} \|G^{1/2} f \|_2^2 \|T\|_F \\ & = \sup_{f:\|f\|_{L^2} = 1}  \int_{\mathbb{T}^d} \mu(x)f(x)^2 \text{d}x \|P\|_F  \leq C_\mu \|T\|_F .
    \end{align*}
    Likewise, we obtain
    \begin{align*}
         \|T\|_{F, \mu}^2 & \geq \inf_{f:\|f\|_{L^2} =1} \|G^{1/2} f\|_{L^2}^2  \|T\|_F \\ & = \inf_{f:\|f\|_{L^2} =1}  \int_{\mathbb{T}^d} \mu(x)f(x)^2 \text{d}x  \|P\|_F \geq c_\mu \|T\|_F.
    \end{align*}
\end{proof}

	\section*{Acknowledgements}
	Both authors are grateful to R. Nickl, K. Abraham and S. Wang for helpful discussions, to Y. Sun for pointing out a mistake in a previous version of the paper and to two anonymous referees for their helpful comments, suggestions and remarks. A. Picard would  like to thank the Statslab and ERC grant UQMSI/647812 for supporting him during the undertaking of this work while visiting R. Nickl’s research group from February to June 2018. 
\bibliography{thesisBibliography}
\bibliographystyle{plain}
\end{document}